\documentclass[a4paper]{amsart}
\usepackage{amssymb, hyperref, aliascnt}

\newtheorem{lma}{Lemma}[section]

\newaliascnt{thmCt}{lma}
\newtheorem{thm}[thmCt]{Theorem}
\aliascntresetthe{thmCt}

\newaliascnt{corCt}{lma}
\newtheorem{cor}[corCt]{Corollary}
\aliascntresetthe{corCt}

\newaliascnt{prpCt}{lma}
\newtheorem{prp}[prpCt]{Proposition}
\aliascntresetthe{prpCt}

\theoremstyle{definition}

\newaliascnt{pgrCt}{lma}
\newtheorem{pgr}[pgrCt]{}
\aliascntresetthe{pgrCt}

\newaliascnt{dfnCt}{lma}
\newtheorem{dfn}[dfnCt]{Definition}
\aliascntresetthe{dfnCt}

\newaliascnt{rmkCt}{lma}
\newtheorem{rmk}[rmkCt]{Remark}
\aliascntresetthe{rmkCt}

\newaliascnt{exaCt}{lma}
\newtheorem{exa}[exaCt]{Example}
\aliascntresetthe{exaCt}

\newaliascnt{qstCt}{lma}
\newtheorem{qst}[qstCt]{Question}
\aliascntresetthe{qstCt}

\newcounter{theoremintro}

\newtheorem{dfnIntro}[theoremintro]{Definition}
\newtheorem{prpIntro}[theoremintro]{Proposition}

\newaliascnt{thmIntroCt}{theoremintro}
\newtheorem{thmIntro}[thmIntroCt]{Theorem}
\aliascntresetthe{thmIntroCt}

\def\today{\number\day\space\ifcase\month\or   January\or February\or
   March\or April\or May\or June\or   July\or August\or September\or
   October\or November\or December\fi\   \number\year}

\newcommand{\NN}{{\mathbb{N}}}

\newcommand{\CC}{{\mathbb{C}}}
\newcommand{\Bdd}{{\mathcal{B}}}
\newcommand{\Cpct}{{\mathcal{K}}}
\newcommand{\ca}{$C^*$-algebra}
\newcommand{\andSep}{\,\,\,\text{ and }\,\,\,}
\newcommand{\axiomO}[1]{(O#1)}
\newcommand{\CatCu}{\ensuremath{\mathrm{Cu}}}
\newcommand{\CuSgp}{$\CatCu$-sem\-i\-group}

\DeclareMathOperator{\Cu}{Cu}
\DeclareMathOperator{\Prim}{Prim}

\title{Nowhere scattered C*-algebras}
\date{\today}

\author{Hannes Thiel}
\address{Hannes Thiel,
Department of Mathematics, Kiel University, Heinrich-Hecht-Platz~6, 24118 Kiel, Germany.}
\email{hannes.thiel@math.uni-kiel.de}
\urladdr{www.hannesthiel.org}

\author{Eduard Vilalta}
\address{Eduard Vilalta,
Departament de Matem\`{a}tiques,
Universitat Aut\`{o}noma de Barcelona,
08193 Bellaterra, Barcelona, Spain}
\email{evilalta@mat.uab.cat}
\urladdr{www.eduardvilalta.com}

\thanks{The first named author was partially supported by the Deutsche Forschungsgemeinschaft (DFG, German Research Foundation) under Germany's Excellence Strategy EXC 2044-390685587 (Mathematics M\"{u}nster: Dynamics-Geometry-Structure) and by the ERC Consolidator Grant No.~681207.
The second named author was partially supported by MINECO (grant No.\ PRE2018-083419 and No.\ PID2020-113047GB-I00), and by the Comissionat per Universitats i Recerca de la Generalitat de Catalunya (grant No.\ 2017SGR01725). 
}

\subjclass[2010]%
{Primary
46L05; % General theory of C*-algebras
Secondary
19K14, % $ K_0$ as an ordered group, traces
46L80, % $K$-theory and operator algebras 
46L85. %Noncommutative topology
}
\date{\today}

\begin{document}

%==========================================================================================
\begin{abstract}
We say that a \ca{} is \emph{nowhere scattered} if none of its quotients contains a minimal open projection.
We characterize this property in various ways, by topological properties of the spectrum, by divisibility properties in the Cuntz semigroup, by the existence of Haar unitaries for states, and by the absence of nonzero ideal-quotients that are elementary, scattered or type~$\mathrm{I}$.

Under the additional assumption of real rank zero or stable rank one, we show that nowhere scatteredness implies even stronger divisibility properties of the Cuntz semigroup.
\end{abstract}

\maketitle

%==========================================================================================
%==========================================================================================
\section{Introduction}

%==========================================================================================
A topological space is said to be \emph{scattered} if each of its nonempty closed subsets contains an isolated point.
Analogously, a \ca{} is \emph{scattered} if each of its nonzero quotients contains a minimal open projection; 
see \autoref{sec:scattered}.
Here, a \emph{minimal open projection} in a \ca{} $A$ is a projection $p\in A$ such that $pAp=\CC p$;
see \autoref{prp:MinOpenProj}.
In this paper, we consider \ca{s} that are very far from scattered:

%==========================================================================================
\begin{dfnIntro}
\label{dfn:algNWS}
A \ca{} is \emph{nowhere scattered} if none of its quotients contains a minimal open projection.
\end{dfnIntro}

%==========================================================================================
Similarly, we say that a topological space is \emph{nowhere scattered} if none of its closed subsets contains an isolated point;
see \autoref{dfn:spaceNWS}.
Thus, one-element sets cannot be closed, and nowhere scattered spaces are far from being $T_1$, let alone Hausdorff.
In \autoref{sec:topological}, we study this topological notion and the relation to its noncommutative counterpart. We prove that a separable \ca{} is nowhere scattered if and only if its spectrum is;
see \autoref{prp:charTopological}.

Besides this topological description, we also show that nowhere scatteredness admits various further characterizations: 
in terms of the structure of ideal-quotients (\autoref{sec:NWS}), 
by the existence of Haar unitaries and maximal abelian subalgebras  with diffuse spectrum (\autoref{sec:diffuse}), 
and by divisibility properties of the Cuntz semigroup (\autoref{sec:EleIdeQuo}).
We present a number of these characterizations in \autoref{thmChar} below.

For \ca{s} of real rank zero, nowhere scatteredness can be described in terms of divisibility properties of the Murray-von Neumann semigroup of projections;
see \autoref{prp:charDivRR0}.
However, in contrast with the Cuntz semigroup, this invariant does not contain enough information to characterize nowhere scatteredness in the general setting;
see \autoref{rmk:charDivRR0}. In \autoref{prp:charDivSR1}, we also provide additional characterizations of nowhere scatteredness for \ca{s} of stable rank one.

%==========================================================================================
\begin{thmIntro}[{\ref{prp:firstChar}, \ref{prp:charDiffuse}, \ref{prp:charDiv}}]
\label{thmChar}
Let $A$ be a \ca.
Then the following are equivalent:
\begin{enumerate}
\item
$A$ is nowhere scattered;
\item
every quotient of $A$ is antiliminal;
\item
$A$ has no nonzero elementary/scattered/type~$\mathrm{I}$ ideal-quotients;
\item
no hereditary sub-\ca{} of $A$ admits a one-dimensional (a finite-dimen\-sional) irreducible representation;
\item
every positive functional (pure state) on $A$ is nowhere scattered;
\item
for every pure state $\varphi$ on $A$ and every ideal $I\subseteq A$, there exists a Haar unitary in $\widetilde{I}$ for $\varphi$ (there exists a masa $C_0(X)\subseteq I$ such that $\varphi$ induces a diffuse measure on $X$ with total mass $\|\varphi|_I\|$);
\item
for every positive functional $\varphi$ on $A$ and every hereditary sub-\ca{} $B\subseteq A$, there exists a Haar unitary in $\widetilde{B}$ for $\varphi$ (there exists a masa $C_0(X)\subseteq B$ such that $\varphi$ induces a diffuse measure on $X$ with total mass $\|\varphi|_B\|$);
\item
every element in $\Cu(A)$ is weakly $(2,\omega)$-divisible;
\item
every element in $\Cu(A)$ is weakly $(k,\omega)$-divisible for every $k\geq 2$.
\end{enumerate}
\end{thmIntro}

%==========================================================================================
In \autoref{sec:permanence}, we study permanence properties of such \ca{s}.

%==========================================================================================
\begin{prpIntro}[{\ref{prp:permHerQuot}, \ref{prp:permLimit}}]
Nowhere scatteredness passes to hereditary sub-\ca{s} (in particular, ideals), to quotients, and to inductive limits.
\end{prpIntro}

%==========================================================================================
Further, nowhere scatteredness satisfies the L{\"o}wenheim-Skolem condition:
For every nowhere scattered \ca{}, the collection of \emph{separable}, nowhere scattered sub-\ca{s} is $\sigma$-complete and cofinal among all separable sub-\ca{s};
see \autoref{prp:LS}.

%==========================================================================================
The concept of nowhere scattered \ca{s} has implicitly appeared in the literature before, but not under this name (or any name for that matter).
For example, our characterizations in terms of divisibility properties of $\Cu(A)$ can also be derived from results in \cite{RobRor13Divisibility};
see \autoref{rmk:RobRor}.
Nowhere scattered, real rank zero \ca{s} have been considered in \cite{PerRor04AFembeddings, EllRor06Perturb}.

The term `nowhere scattered' was first used in \cite{Thi20arX:diffuseHaar} to name positive functionals on \ca{s} that give no weight to scattered ideal-quotients. 
It is the purpose of this paper to initiate a systematic study of nowhere scatteredness by collecting, systematizing and complementing existing results.

%==========================================================================================
\subsection*{Terminology and notation}
An ideal in a \ca{} means a closed, two-sided ideal.
Given a \ca{} $A$, we use $A_+$ to denote the positive elements in $A$.
%Given $a\in A$, we use $\spec(a)$ to denote its spectrum, and we let $\supp(a)$ denote the support projection in $A^{**}$.
Given a Hilbert space $H$, we let $\Bdd(H)$ denote the \ca{} of bounded, linear operators on $H$, and $\Cpct(H)$ is the ideal of compact operators.

%==========================================================================================
%==========================================================================================
\section{Scattered C*-algebras}
\label{sec:scattered}

%==========================================================================================
Following \cite[Definition~2.1]{Jen77ScatteredCAlg}, a \ca{} is said to be \emph{scattered} if each of its states is atomic, that is, a countable weighted sum of pure states.
This definition is inspired by a result of Pe{\l}czy\'{n}ski and Semadeni, \cite{PelSem59SpCtsFcts3}, which states that a compact, Hausdorff space $X$ is scattered if and only if every regular Borel probability measure on $X$ is atomic. 
%Over the years, different authors found various characterizations of scatteredness, which proved useful in applications.
Scatteredness admits various characterizations, both for topological spaces and for \ca{s} -- and several of these descriptions are remarkably similar in the two settings;
see \autoref{pgr:scattered}. % below, we list the most important ones.

Note that a locally compact, Hausdorff space $X$ is scattered if and only if its one-point compactification is, and this is in turn equivalent to $C_0(X)$ being a scattered \ca.
Thus, it follows from \cite[Lemma~2.2]{Kus12everySubAF} that a \ca{} $A$ is scattered if and only if for each commutative sub-\ca{} $C_0(X)\subseteq A$ the space $X$ is scattered.

%==========================================================================================
\begin{pgr}
\label{prp:MinOpenProj}
By a \emph{minimal open projection} in a \ca{} $A$ we mean a nonzero projection $p\in A$ such that $pAp=\CC p$.
In some parts of the literature, such projections are simply called `minimal projections', for example \cite[Deﬁnition~1.1]{GhaKos18NCCantorBendixson}.
We do not follow this convention since it could potentially lead to confusion: one could understand a minimal projection in a \ca{} to be a nonzero projection $p$ such that every subprojection $q$ of $p$ satisfies $q=0$ or $q=p$, and this would include projections like the unit of the Jiang-Su algebra.

Let us justify our terminology:
An \emph{open projection} in a \ca{} $A$ is a projection $p\in A^{**}$ such that there exists an increasing net in $A_+$ that converges in the weak*-topology of $A^{**}$ to $p$.
Then $pA^{**}p\cap A$ is a hereditary sub-\ca{} of $A$, and this induces a natural bijection between open projections in $A$ and hereditary sub-\ca{s} of $A$.
For details we refer to \cite[p.77f]{Ped79CAlgsAutGp}.

Claim:
\emph{Let $p\in A^{**}$ be an open projection.
Then $p$ belongs to $A$ and satisfies $pAp=\CC p$ if and only if $p$ is minimal in the sense that for every open projection $q\in A^{**}$ with $q\leq p$ we have $q=0$ or $q=p$.}
Indeed, the forward implication is clear.
Conversely, assume that $p$ is minimal in the above sense and set $B:=pA^{**}p\cap A$.
If $x,y\in B$ satisfy $xy=0$, then $x^*x$ and $yy^*$ are orthogonal, positive elements, whose support projections are orthogonal open subprojections of $p$.
Using minimality of~$p$, we deduce that $x^*x=0$ or $yy^*=0$, which shows that $B$ has no zero divisors and thus $B\cong\CC$ by the Gelfand-Mazur theorem.
This implies that $B$ is unital and thus $p\in A$, and also $pAp=\CC p$.
\end{pgr}

%==========================================================================================
\begin{pgr}
\label{pgr:scattered}
Let $A$ be a \ca{} ($X$ be a locally compact, Hausdorff space). Then, the following are equivalent:
\begin{enumerate}
\item
$A$ is \emph{scattered}, that is, every state on $A$ is atomic;

\noindent 
(every regular Borel probability measure on $X$ is atomic, 
\cite{PelSem59SpCtsFcts3})
\item
$A^{**}$ is an atomic von Neumann algebra, that is, $A^{**}$ is isomorphic to a product of type~$\mathrm{I}$ factors,
\cite[Theorem~2.2]{Jen77ScatteredCAlg};
\item
$A^*$ has the Radon-Nikod\`{y}m property,
\cite[Theorem~3]{Chu81ScatteredRNP};
\item
for every separable sub-\ca{} $B\subseteq A$, the dual space $B^*$ is separable -- to see this, note that~(3) implies~(4) by \cite[Theorem~1]{Chu81ScatteredRNP}, and that~(4) implies~(12) since $C_0((0,1])$ is a separable \ca{} with nonseparable dual space;
\item
$A$ admits a composition series $(I_\lambda)_\lambda$ such that each successive quotient $I_{\lambda+1}/I_\lambda$ is elementary,
\cite[Theorem~2]{Jen78ScatteredCAlg2};
\item
$A$ is type~$\mathrm{I}$ and $\widehat{A}$ is scattered,
\cite[Corollary~3]{Jen78ScatteredCAlg2};
\item
$A$ is type~$\mathrm{I}$ and the center of $A$ is scattered,
\cite[Theorem~2.2]{Kus10CharScattered};
\item
every sub-\ca{} of $A$ has real rank zero,
\cite[Theorem~2.3]{Kus12everySubAF};

\noindent
(every continuous image of $X$ is zero-dimensional, \cite{PelSem59SpCtsFcts3});
\item
every nonzero sub-\ca{} of $A$ contains a minimal open projection, \cite[Theorem~1.4]{GhaKos18NCCantorBendixson};
\item
the spectrum of every self-adjoint element in $A$ is countable,
\cite{Hur78SpectralCharClassCAlg};

\noindent
(every continuous function $X\to\mathbb{R}$ has countable range,
\cite[Corollary~8.5.6]{Sem71BSpCtsFcts})
\item
every quotient of $A$ contains a minimal open projection,
\cite[Theorem~1.4]{GhaKos18NCCantorBendixson};

\noindent
($X$ is \emph{scattered}, that is, every closed subset of $X$ contains an isolated point)
\item
there exists no sub-\ca{} $C_0((0,1])\subseteq A$, \cite[Theorem~1.4]{GhaKos18NCCantorBendixson};

\noindent
(there is no continuous, surjective map $X\to[0,1]$,
\cite{PelSem59SpCtsFcts3})
\end{enumerate}
\end{pgr}

%==========================================================================================
%==========================================================================================
\section{Nowhere scattered C*-algebras}
\label{sec:NWS}

%==========================================================================================
In this section we prove basic characterizations of nowhere scatteredness;
see \autoref{prp:firstChar}.
We observe that (weakly) purely infinite \ca{s} are nowhere scattered, and we show that a von Neumann algebra is nowhere scattered if and only if its type~$\mathrm{I}$ summand is zero; see \autoref{exa:wpi} and \autoref{prp:VNA} respectively.

%==========================================================================================
A \ca{} is \emph{elementary} if it is isomorphic to the algebra of compact operators on some Hilbert space.
An \emph{ideal-quotient} of a \ca{} $A$ is a (closed, two-sided) ideal of a quotient of $A$.
A relatively open subset of a closed subset of a topological space is said to be \emph{locally closed}.
Using the correspondence between ideals (quotients) of $A$ and open (closed) subsets of its primitive ideal space $\Prim(A)$, it follows that ideal-quotients of $A$ correspond to locally closed subsets of $\Prim(A)$.

An element $a\in A_+$ is \emph{abelian} if the heredtiary sub-\ca{} $\overline{aAa}$ is commutative.
A \ca{} is \emph{type~$\mathrm{I}$} if every nonzero quotient contains a nonzero, abelian element;
see \cite[Definition~IV.1.1.6]{Bla06OpAlgs}.
A \ca{} is \emph{antiliminal} if it contains no nonzero abelian positive elements;
see \cite[Definition~IV.1.1.6]{Bla06OpAlgs}.
An irreducible representation $\pi\colon A\to\Bdd(H)$ is said to be \emph{GCR} if $\pi(A)\cap\Cpct(H)\neq\{0\}$.

%==========================================================================================
\begin{thm}
\label{prp:firstChar}
Let $A$ be a \ca.
Then the following are equivalent:
\begin{enumerate}
\item
$A$ is nowhere scattered
(no quotient contains a minimal open projection);
\item
every quotient of $A$ is antiliminal;
\item
$A$ has no nonzero ideal-quotients of type~$\mathrm{I}$;
\item
$A$ has no nonzero scattered ideal-quotients;
\item
$A$ has no nonzero elementary ideal-quotients;
\item
$A$ has no nonzero irreducible GCR representation;
\item
no hereditary sub-\ca{} of $A$ admits a finite-dimensional irreducible representation;
\item
no hereditary sub-\ca{} of $A$ admits a one-dimensional irreducible representation.
\end{enumerate}
\end{thm}
\begin{proof}
Let us first prove that~(1) implies~(2). Assume that there exists an ideal $I\subseteq A$ such that $A/I$ is not antiliminal.
Set $B:=A/I$ and choose a nonzero, abelian element $b\in B_+$.
This allows us to obtain an ideal $J\subseteq\overline{bBb}$ such that $\overline{bBb}/J\cong\CC$.
Let $K\subseteq B$ be the ideal generated by $J$.
Since $\overline{bBb}\subseteq B$ is hereditary, we have $K\cap\overline{bBb}=J$.
It follows that the image of $\overline{bBb}$ under the quotient map $B\to B/K$ is isomorphic to $\overline{bBb}/J$, and so $B/K$ contains a minimal open projection, which contradicts~(1), as required.

To show that~(2) implies~(3), let $I\subseteq J\subseteq A$ be ideals such that $J/I$ is nonzero and of type~$\mathrm{I}$.
Then $J/I$ contains a nonzero, abelian element, whence $A/I$ is not antiliminal.
The implications `(3)$\Rightarrow$(4)$\Rightarrow$(5)' follow using that every elementary \ca{} is scattered, and that every scattered \ca{} is type~$\mathrm{I}$.

Now assume that there exists a nonzero irreducible GCR representation $\pi\colon A\to\Bdd(H)$.
Then $\Cpct(H)\subseteq\pi(A)$ by \cite[Corollary~IV.1.2.5]{Bla06OpAlgs}.
Let $I$ be the kernel of $\pi$, and set $J:=\pi^{-1}(\Cpct(H))$.
Then $I\subseteq J\subseteq A$ are ideals such that $J/I\cong\Cpct(H)$. This proves  that~(5) implies~(6).

To see that~(6) implies~(7), let $B\subseteq A$ be a hereditary sub-\ca{}, and let $\pi_0\colon B\to\Bdd(H_0)$ be a nonzero, finite-dimensional irreducible representation.
Extend $\pi_0$ to an irreducible representation $\pi\colon A\to\Bdd(H)$ on some Hilbert space $H$ containing $H_0$, such that $H_0$ is invariant under $\pi(B)$ and such that $\pi(b)\xi=\pi_0(b)\xi$ for all $b\in B$ and $\xi\in H_0$;
see \cite[Proposition~II.6.4.11]{Bla06OpAlgs}.

Let $H_1$ be the essential subspace of $\pi|_B$, that is, the closed linear subspace of $H$ generated by $\pi(B)H$.
By \cite[Proposition~II.6.1.9]{Bla06OpAlgs}, the restriction of $\pi|_B$ to $H_1$ is irreducible.
Since $H_0$ is a closed subspace of $H_1$ that is invariant under $\pi(B)$, it follows that $H_0=H_1$.
Let $p\in\Bdd(H)$ denote the orthogonal projection onto $H_0$.
Then $p\pi(b)=\pi(b)$ and therefore $\pi(b)=\pi(b)p=p\pi(b)p$ for every $b\in B$.
Thus, $\pi(B)\subseteq\Cpct(H)$, and it follows that $\pi$ is GCR.

The implication `(7)$\Rightarrow$(8)' is clear. 
Finally, to see that~(8) implies~(1), let $I\subseteq A$ be an ideal and let $p\in A/I$ be a minimal open projection.
Let $\pi\colon A\to A/I$ be the quotient map.
Set $B:=\pi^{-1}(\CC p)$.
Then $B$ is a hereditary sub-\ca{} of $A$ that admits a nonzero, one-dimensional representation.
\end{proof}

%==========================================================================================
\begin{exa}\label{exa:Sim}
A simple \ca{} is nowhere scattered if and only if it is not elementary.
In particular, a unital simple \ca{} is nowhere scattered if and only it is infinite-dimensional.
\end{exa}

%==========================================================================================
A \ca{} $A$ is \emph{purely infinite} (in the sense of Kirchberg-R{\o}rdam, \cite[Definition~4.1]{KirRor00PureInf}) if and only if every element $x$ of its Cuntz semigroup  $\Cu (A)$, as defined in \autoref{pgr:Cu}, satisfies $2x=x$;
see \cite[Theorem~4.16]{KirRor00PureInf} and \cite[Proposition~7.2.8]{AntPerThi18TensorProdCu}.

More generally, $A$ is said to be \emph{weakly purely infinite} if there exists $n\in\NN$ such that every $x\in \Cu(A)$ satisfies $2nx=nx$;
see \cite[Definition~4.3]{KirRor02InfNonSimpleCalgAbsOInfty}.
Every purely infinite \ca{} is weakly purely infinite, and it is an open problem if the converse holds.

A \ca{} $A$ is \emph{traceless} if no algebraic ideal of $A$ admits a nonzero quasitrace;
see \cite[Definition~4.2]{KirRor02InfNonSimpleCalgAbsOInfty}.
Equivalently, every lower-semicontinuous $2$-quasitrace $A_+\to[0,\infty]$ (in the sense of \cite{EllRobSan11Cone}) takes only values in $\{0,\infty\}$;
see \cite[Remark~2.27(viii)]{BlaKir04PureInf}.
Using the correspondence between quasitraces on $A$ and functionals on $\Cu(A)$, this is also equivalent to the property that every $x\in\Cu(A)$ satisfies $2\widehat{x}=\widehat{x}$.
Consequently, every weakly purely infinite \ca{} is traceless;
see \cite[Theorem~4.8]{KirRor02InfNonSimpleCalgAbsOInfty}.

%==========================================================================================
\begin{exa}
\label{exa:wpi}
Since every elementary \ca{} admits a nontrivial quasitrace, it follows that a traceless \ca{} cannot have nonzero elementary ideal-quotients. Using \autoref{prp:firstChar}, we deduce that every traceless \ca{} is nowhere scattered.
In particular, every (weakly) purely infinite \ca{} is nowhere scattered.
\end{exa}

%==========================================================================================
\begin{prp}
\label{prp:VNA}
A von Neumann algebra is nowhere scattered if and only if its type~$\mathrm{I}$~summand is zero.
\end{prp}
\begin{proof}
Let $M$ be a von Neumann algebra. 
To show the forward implication, assume that $M$ is nowhere scattered.
If $p\in M$ is a nonzero abelian projection, then the hereditary sub-\ca{} $pMp$ is commutative and therefore admits a one-dimensional irreducible representation.
Thus, it follows from \autoref{prp:firstChar} that $M$ contains no nonzero abelian projections, and thus its type~$\mathrm{I}$~summand is zero.

To show the converse implication, assume that the type~$\mathrm{I}$~summand of $M$ is zero.
To reach a contradiction, assume that $B\subseteq M$ is a hereditary sub-\ca{} that admits a finite-dimensional, irreducible representation $\pi$;
see \autoref{prp:firstChar}.
Choose a projection $p\in B$ such that $\pi(p)\neq 0$.
Since $M$ has no type~$\mathrm{I}$~summand, for each $n\geq 1$ we can find $2^n$ pairwise orthogonal and equivalent projections whose sum is equal to~$p$;
see \cite[Proposition~V.1.35]{Tak02ThyOpAlgs1}.
This contradicts that $\pi(p)$ has nonzero, finite rank.
\end{proof}

%==========================================================================================
%==========================================================================================
\section{Permanence properties}
\label{sec:permanence}

%==========================================================================================
In this section, we show that nowhere scatteredness enjoys many permanence properties.
In particular, it passes to hereditary sub-\ca{s} (hence, ideals), quotients, and inductive limits.

%==========================================================================================
\begin{prp}
\label{prp:permHerQuot}
Let $A$ be a nowhere scattered \ca.
Then every quotient and every hereditary sub-\ca{} of $A$ is nowhere scattered.
\end{prp}
\begin{proof}
This follows from \autoref{prp:firstChar}, since condition~(2) passes to quotients, and condition~(7) passes to hereditary sub-\ca{s}.
\end{proof}

%==========================================================================================
\begin{prp}
\label{prp:permExt}
Let $A$ be a \ca, and let $I\subseteq A$ be an ideal.
Then $A$ is nowhere scattered if (and only if) $I$ and $A/I$ are nowhere scattered.
\end{prp}
\begin{proof}
Assume that $I$ and $A/I$ are nowhere scattered.
To verify condition~(4) of \autoref{prp:firstChar}, let $J\subseteq K\subseteq A$ be ideals such that $K/J$ is scattered.
Then $I\cap J \subseteq I\cap K\subseteq I$ are ideals such that $(I\cap K)/(I\cap J)$ is isomorphic to an ideal of $K/J$ and is therefore scattered.
Since $I$ is nowhere scattered, we get $I\cap J=I\cap K$.

Similarly, $J/I\cap J\subseteq K/I\cap K\subseteq A/I$ are ideals such that $(J/I\cap J)/(K/I\cap K)$ is isomorphic to a quotient of $J/K$ and therefore scattered.
Since $A/I$ is nowhere scattered, we get $J/I\cap J=K/I\cap K$.
It follows that $J=K$.
\end{proof}

%==========================================================================================
A \ca{} is scattered if and only if each of its sub-\ca{s} has real rank zero; see \autoref{pgr:scattered}.
It follows that every sub-\ca{} of a scattered \ca{} is again scattered. 
The analog for nowhere scatteredness does not hold:
Take, for example, $\CC\subseteq A$ in any unital, nowhere scattered \ca{} $A$.

Conversely, given a full sub-\ca{} $B\subseteq A$, it is natural to ask if $A$ is nowhere scattered whenever $B$ is.
Without extra assumptions, this fails:
Consider for example a type~$\mathrm{II}$ factor $M\subseteq\Bdd(H)$.
By \autoref{prp:VNA}, $M$ is nowhere scattered, but $\Bdd(H)$ contains an elementary ideal and therefore is \emph{not} nowhere scattered.
The `right' additional condition is for $B$ to \emph{separate the ideals} of $A$, that is, two ideals $I,J\subseteq A$ satisfy $I=J$ whenever $I\cap B=J\cap B$.

%==========================================================================================
\begin{prp}
\label{prp:separatingIdeals}
Let $A$ be \ca, and let $B\subseteq A$ be a nowhere scattered sub-\ca{} that separates the ideals of $A$.
Then $A$ is nowhere scattered.
\end{prp}
\begin{proof}
Let $I\subseteq J\subseteq A$ be ideals with $J/I$ scattered. In order to prove condition~(4) in \autoref{prp:firstChar}, we show that this ideal-quotient is zero. 

Note that $(J\cap B)/(I\cap B)$ is scattered since it is isomorphic to a subalgebra of $J/I$.
Since $B$ is nowhere scattered, it follows that $J\cap B=I\cap B$, and thus $J=I$, as desired.
\end{proof}

%==========================================================================================
\begin{exa}
Consider an action of an exact, discrete group $G$ on a nowhere scattered \ca{} $A$. 
%We consider the canonical embedding of $A$ into the reduced crossed product $A\rtimes_\red G$.
Assume that the induced action $G\curvearrowright \widehat{A}$ is essentially free.
Then, by \cite[Theorem~1.20]{Sie10IdealsRedCrProd}, $A$ separates the ideals of the reduced crossed product $A\rtimes_{\mathrm{red}} G$.
Hence, $A\rtimes_{\mathrm{red}} G$ is nowhere scattered by \autoref{prp:separatingIdeals}.
\end{exa}

%==========================================================================================
Given a \ca{} $A$, recall that a family $(A_\lambda)_{\lambda\in\Lambda}$ of sub-\ca{s} $A_\lambda\subseteq A$ is said to \emph{approximate $A$} if for every finite subset $\{a_1,\ldots,a_n\}\subseteq A$ and $\varepsilon>0$ there exists $\lambda\in\Lambda$ and $b_1,\ldots,b_n\in A_\lambda$ such that $\|a_k-b_k\|<\varepsilon$ for $k=1,\ldots,n$.

%==========================================================================================
\begin{prp}
\label{prp:permApprox}
Let $A$ be \ca, and let $(A_\lambda)_{\lambda\in\Lambda}$ be a family of nowhere scattered sub-\ca{s} of $A$ that approximates $A$.
Then $A$ is nowhere scattered.
\end{prp}
\begin{proof}
As above, let us prove condition~(4) of \autoref{prp:firstChar}. Take $I\subseteq J\subseteq A$ ideals such that $J/I$ is scattered.
For each $\lambda$, the ideal-quotient $(A_\lambda\cap J)/(A_\lambda\cap I)$ of $A_\lambda$ is isomorphic to a subalgebra of $J/I$ and therefore scattered.
Since $A_\lambda$ is nowhere scattered, we get $A_\lambda\cap J=A_\lambda\cap I$.

To verify that $J=I$, let $a\in J_+$.
Given $\varepsilon>0$, use that $(A_\lambda)_{\lambda\in\Lambda}$ approximates~$A$ and a functional calculus argument to obtain $\lambda$ and a positive element $b\in A_\lambda$ such that $\|a-b\|<\varepsilon$.
Consider the $\varepsilon$-cut-down $(b-\varepsilon)_+$ as in \autoref{pgr:Cu}.
Then there exists $r\in A$ such that $(b-\varepsilon)_+ = rar^*$;
see, for example, \cite[Lemma~2.29]{Thi17:CuLectureNotes}.
Since $J$ is an ideal, and since $A_\lambda$ is closed under functional calculus, we deduce that $(b-\varepsilon)_+$ belongs to $A_\lambda\cap J$, and thus to $I$.
We have $\|a-(b-\varepsilon)_+\|<2\varepsilon$, and since $\varepsilon>0$ was arbitary, we deduce that $a\in I$, as desired.
\end{proof}

%==========================================================================================
\begin{prp}
\label{prp:permLimit}
An inductive limit of nowhere scattered \ca{s} is nowhere scattered.
\end{prp}
\begin{proof}
Let $\Lambda$ be a directed set, let $(A_\lambda)_{\lambda\in\Lambda}$ be a family of nowhere scattered \ca{s}, and let $\varphi_{\kappa,\lambda}\colon A_\lambda\to A_\kappa$ be coherent connecting morphisms for $\lambda\leq\kappa$ in~$\Lambda$.
The inductive limit of this system is given by a \ca{} $A$ together with morphisms into the limit $\varphi_{\lambda}\colon A_\lambda\to A$.
For each $\lambda$, set $B_\lambda:=\varphi_\lambda(A_\lambda)\subseteq A$.
Then $B_\lambda$ is a quotient of $A_\lambda$, and therefore is nowhere scattered by \autoref{prp:permHerQuot}.
It follows from standard properties of the inductive limit that $(B_\lambda)_\lambda$ approximates $A$.
Hence, $A$ is nowhere scattered by \autoref{prp:permApprox}.
\end{proof}

%==========================================================================================
\begin{lma}
\label{prp:CharNoOneDimRepr}
Let $A$ be a \ca{}, and let $a\in A_+$.
Then $\overline{aAa}$ has no one-dimensional irreducible representations if and only if there exists a countable subset $G\subseteq A$ such that $(axa)^2=0$ for each $x\in G$, and such that $a\in C^*(\{axa:x\in G\})$.
\end{lma}
\begin{proof}
Without loss of generality, we may assume that $a$ is strictly positive and thus $A=\overline{aAa}$.
To show the backward implication, assume that $G$ has the stated properties, and let $\pi\colon A\to\CC$ be a one-dimensional representation.
Given $x\in G$, we have $\pi(axa)^2=\pi((axa)^2)=0$, and therefore $\pi(axa)=0$.
It follows that $\pi(a)=0$, and thus $\pi=0$, as desired.

To show the forward implication, assume that $A$ has no one-dimensional irreducible representations.
Set $N_2:=\{y\in A : y^2=0\}$.
An additive commutator in~$A$ is an element of the form $cd-dc$ for some $c,d\in A$.
Let $L$ denote the linear span of the additive commutators in~$A$.
Then, $A$ is generated by $L$ as a (closed, two-sided) ideal.
By \cite[Theorem~1.3]{Rob16LieIdeals}, the \ca{} generated by $L$ agrees with the (closed, two-sided) ideal generated by~$L$.
Hence, $A=C^*(L)$.
By \cite[Corollary~2.3]{Rob16LieIdeals}, the closure of~$L$ agrees with the closed, linear span of $N_2$, which implies that $A=C^*(N_2)$.
This allows us to choose a countable subset $F\subseteq N_2$ such that $a\in C^*(F)$.
The statement now follows from the following claim:

\textbf{Claim:} \emph{Let $y\in N_2$ and $\gamma>0$.
Then there exists $x\in A$ such that $axa\in N_2$ and $\|y-axa\|<\gamma$.}

We may assume that $\|y\|\leq 1$.
By \cite[Theorem~5]{Shu08LiftingNilpotent}, the universal \ca{} generated by a contractive, square-zero element is projective.
This implies that the corresponding relations are stable (see for example \cite[Theorem~14.1.4]{Lor97LiftingSolutions}), that is, for every $\varepsilon>0$ there exists $\delta=\delta(\varepsilon)>0$ such that if $z$ is an element in a \ca{} satisfying $\|z\|\leq 1$ and $\|z^2\|\leq\delta$, then there exists a contractive, square-zero element $z'$ with $\|z-z'\|<\varepsilon$.

For each $n\geq 1$, let $f_n\colon\mathbb{R}\to[0,1]$ be a continuous function that takes the value $0$ on $(-\infty,\tfrac{1}{n}]$ and that takes the value $1$ on $[\tfrac{2}{n},\infty)$.
Then $(f_n(a))_n$ is an approximate unit for $A$.
Thus, $\lim_{n\to\infty}\|(f_n(a)yf_n(a))^2\|=0$.
Thus, we can find $m\geq 1$ such that
\[
\|(f_m(a)yf_m(a))^2\|<\delta(\tfrac{\gamma}{2}), \andSep
\|y-f_m(a)yf_m(a)\|<\tfrac{\gamma}{2}.
\]

We view $f_m(a)yf_m(a)$ as an element of the sub-\ca{} $\overline{f_m(a)Af_m(a)}$.
Using also that $f_m(a)yf_m(a)$ is contractive, we obtain $z\in\overline{f_m(a)Af_m(a)}$ such that $z^2=0$ and $\|z-f_m(a)yf_m(a)\|<\tfrac{\gamma}{2}$.
Then $\|y-z\|<\gamma$.

It remains to show that $z=axa$ for some $x\in A$.
Note that $f_{2m}(a)$ acts as a unit on $\overline{f_m(a)Af_m(a)}$.
Using functional calculus, choose $b\in A$ such that $f_{2m}(a)=ab=ba$.
We obtain
\[
z 
= f_{2m}(a)zf_{2m}(a)
= abzba,
\]
and thus $x:=bzb$ has the claimed properties.
\end{proof}

%==========================================================================================
\begin{prp}
\label{prp:NWS-PreLocal}
Let $A$ be a \ca.
Then the following are equivalent:
\begin{enumerate}
\item
$A$ is nowhere scattered;
%\item
%for every $a\in A_+$ there exists sequence $(x_n)_n$ in $A$ such that $(ax_na)^2=0$ for each $n$ and such that $a\in C^*(\{ax_na:x\in\NN\})$.
\item
for every $a\in A_+$ there exists $b\in A_+$ with $\|a-b\|<\tfrac{1}{2}$ and such that $\overline{(b-\tfrac{1}{2})_+A(b-\tfrac{1}{2})_+}$ has no one-dimensional irreducible representations;
\item
for every $a\in A_+$ there exist $b\in A_+$ with $\|a-b\|<\tfrac{1}{2}$ and a countable subset $G\subseteq A$ such that $((b-\tfrac{1}{2})_+x(b-\tfrac{1}{2})_+)^2=0$ for each $x\in G$ and such that $(b-\tfrac{1}{2})_+\in C^*(\{(b-\tfrac{1}{2})_+x(b-\tfrac{1}{2})_+:x\in G\})$.
\end{enumerate}
\end{prp}
\begin{proof}
If $A$ is nowhere scattered, then by \autoref{prp:firstChar} every hereditary sub-\ca{} of $A$ has no one-dimensional irreducible representations.
This shows that~(1) implies~(2).
The equivalence between~(2) and~(3) follows from \autoref{prp:CharNoOneDimRepr}.

It remains to show that~(2) implies~(1). 
Thus, assume that~(2) holds, and note that it suffices to verify statement~(8) of \autoref{prp:firstChar}.
So let $B\subseteq A$ be a hereditary sub-\ca{}.
To reach a contradiction, let $\pi\colon B\to\CC$ be a one-dimensional irreducible representation.
Extend $\pi$ to an irreducible representation $\tilde{\pi}$ of $A$ on some Hilbert space $H$, such that there is a one-dimensional subspace $H_0\subseteq H$ that is invariant under $\tilde{\pi}(B)$, and such that $\tilde{\pi}|_B$ agrees with $\pi$ on $H_0$.

Choose $a\in B_+$ with $\pi(a)=1$.
By assumption, we obtain $b\in A_+$ such that $\|a-b\|<\tfrac{1}{2}$, and such that $\overline{(b-\tfrac{1}{2})_+A(b-\tfrac{1}{2})_+}$ has no one-dimensional irreducible representations.
By \cite[Lemma~2.2]{KirRor02InfNonSimpleCalgAbsOInfty}, there exists $y\in A$ such that $(b-\tfrac{1}{2})_+=yay^*$.
Set $x:=ya^{1/2}$.
Then 
\[
(b-\tfrac{1}{2})_+ = xx^*, \andSep
x^*x\in\overline{aAa}\subseteq B.
\]

This implies that $\overline{x^*Ax}$ is isomorphic (as a \ca{}) to $\overline{xAx^*}$.
We have $\overline{xAx^*}=\overline{(b-\tfrac{1}{2})_+A(b-\tfrac{1}{2})_+}$, which does not have one-dimensional irreducible representations.
It follows that $\overline{x^*Ax}$ has no one-dimensional irreducible representations, and thus $\pi$ vanishes on $\overline{x^*Ax}$.
Hence, $\pi(x^*x)=0$, and $\tilde{\pi}(x^*x)=0$.
This implies $\tilde{\pi}(x)=0$, and so
\[
\tilde{\pi}\big( (b-\tfrac{1}{2})_+ \big)
= \tilde{\pi}(xx^*)
= 0.
\]

But
\[
\| a - (b-\tfrac{1}{2})_+ \|
\leq \| a-b \| + \| b - (b-\tfrac{1}{2})_+\|
< 1,
\]
and thus $\|\pi(a)\| = \|\tilde{\pi}(a)\|<1$, a contradiction.
\end{proof}

%==========================================================================================
\begin{rmk}
In condition~(2) of \autoref{prp:NWS-PreLocal}, it is not enough to require $\overline{bAb}$ to have no one-dimensional irreducible representations.
Indeed, in any unital \ca{} $A$ without one-dimensional irreducible representations (for example $A=M_2$), for every $a\in A_+$ and $\varepsilon>0$ we can set $b=a+\tfrac{\varepsilon}{2}\in A_+$, which satisfies $\|a-b\|<\varepsilon$ and $\overline{bAb}=A$.
\end{rmk}

%==========================================================================================
\begin{prp}
\label{prp:separablyInheritable}
Let $A$ be a nowhere scattered \ca{}, and let $B_0\subseteq A$ be a separable sub-\ca{}.
Then there exists a separable, nowhere scattered sub-\ca{} $B\subseteq A$ such that $B_0\subseteq B$.
\end{prp}
\begin{proof}
We will inductively choose:
\begin{itemize}
\item
an increasing sequence $B_0\subseteq B_1\subseteq \ldots$ of separable sub-\ca{s} of $A$;
% $B_k\subseteq A$ for $k\geq 0$;
\item
for each $k\geq 0$, a countable, dense subset $F_k\subseteq(B_k)_+$;
\item
for each $b\in F_k$ a countable set $G_{k,b}\subseteq A$ such that $((b-\tfrac{1}{2})_+x(b-\tfrac{1}{2})_+)^2=0$ for every $x\in G_{k,b}$, and such that $(b-\tfrac{1}{2})_+$ belongs to the sub-\ca{} generated by $\{(b-\tfrac{1}{2})_+x(b-\tfrac{1}{2})_+ : x\in G_{k,b}\}$.
\end{itemize}
We will also ensure that $G_{k,b}\subseteq B_{k+1}$ for each $k\geq 0$ and $b\in F_k$.

Let $k\geq 0$, and assume that we have chosen $B_k$.
We will describe how to choose $F_k$, $G_{k,b}$ and $B_{k+1}$.
First, let $F_k$ be any countable, dense subset of $(B_k)_+$.
For each $b\in F_k$, using that $\overline{(b-\tfrac{1}{2})_+A(b-\tfrac{1}{2})_+}$ has no irreducible one-dimensional representations (since $A$ is nowhere scattered), we can apply \autoref{prp:CharNoOneDimRepr} to obtain a countable subset $G_{k,b}\subseteq A$ with the claimed properties.
Then let $B_{k+1}$ be the sub-\ca{} of $A$ generated by $B_k$ together with $\bigcup_{b\in F_k}G_{k,b}$.

Set $B:=\overline{\bigcup_k B_k}$.
Then $B$ is a separable sub-\ca{} of $A$ containing $B_0$.
To see that $B$ is nowhere scattered, we verify statement~(3) of \autoref{prp:NWS-PreLocal}.
So let $a\in B_+$.
Using that $\bigcup_k F_k$ is dense in $B_+$, we can find $k$ and $b\in F_k$ such that $\|a-b\|<\tfrac{1}{2}$.
By construction, $B$ contains the set $G_{k,b}$, which satisfies the desired conditions.
\end{proof}

%==========================================================================================
A property $\mathcal{P}$ for \ca{s} satisfies the \emph{L{\"o}wenheim-Skolem condition} if for every \ca{} $A$ with property $\mathcal{P}$ there exists a family $\mathcal{S}$ of separable sub-\ca{s} of $A$ that each have property $\mathcal{P}$, and such that $\mathcal{S}$ is $\sigma$-complete (for every countable, directed subfamily $\mathcal{D}\subseteq\mathcal{S}$, the \ca{} $\overline{\bigcup\mathcal{D}}\subseteq A$ belongs to $\mathcal{S}$) and cofinal (for every separable sub-\ca{} $B_0\subseteq A$ there exists $B\in\mathcal{S}$ with $B_0\subseteq B$).

It is known that many interesting properties of \ca{s} (such as real rank zero, stable rank one, nuclearity, simplicity) satisfy the L{\"o}wenheim-Skolem condition.

For properties of \CuSgp{s}, the L{\"o}wenheim-Skolem condition was considered in \cite{ThiVil21DimCu2}, where it was also shown that properties like \axiomO{5}, \axiomO{6} and weak cancellation each satisfy it;
see Sections \ref{sec:O8} and \ref{sec:rr0sr1} for definitions.

%==========================================================================================
\begin{prp}
\label{prp:LS}
Let $A$ be a nowhere scattered \ca.
Then the family $\mathcal{S}$ of separable, nowhere scattered sub-\ca{s} of $A$ is $\sigma$-complete and cofinal.
In particular, nowhere scatteredness satisfies the L{\"o}wenheim-Skolem condition.
\end{prp}
\begin{proof}
%Let $\mathcal{S}$ be the family of separable, nowhere scattered sub-\ca{s} of $A$.
To show that $\mathcal{S}$ is $\sigma$-complete, let $\mathcal{D}\subseteq\mathcal{S}$ be a countable, directed subfamily. 
Then $\overline{\bigcup\mathcal{D}}$ is the inductive limit of $\mathcal{D}$, considered as a net  indexed over itself.
Hence, it follows from \autoref{prp:permLimit} that $\overline{\bigcup\mathcal{D}}$ is nowhere scattered and thus belongs to~$\mathcal{S}$, as desired.
Further, by \autoref{prp:separablyInheritable}, $\mathcal{S}$ is cofinal.
\end{proof}

%==========================================================================================
\begin{prp}
\label{prp:Morita}
Let $A$ and $B$ be Morita equivalent \ca{s}.
Assume that $A$ is nowhere scattered.
Then so is $B$.

In particular, a \ca{} $D$ is nowhere scattered if and only if its stabilization $D\otimes\Cpct$ is.
\end{prp}
\begin{proof}
By \cite[Theorem~II.7.6.9]{Bla06OpAlgs}, $A$ and $B$ are isomorphic to complementary full corners in another \ca{}, that is, there exists a \ca{} $C$ and a projection $p\in M(C)$ such that $pCp$ and $(1-p)C(1-p)$ are full (hereditary) sub-\ca{s} of $C$ satisfying $A\cong pCp$ and $B\cong(1-p)C(1-p)$.

Assuming that $A$ is nowhere scattered, it follows that $pCp$ is as well.
Since~$pCp$ is a full hereditary sub-\ca{}, it separates the ideals of~$C$.
Hence, $C$ is nowhere scattered by \autoref{prp:separatingIdeals}.
Using that $(1-p)C(1-p)$ is a hereditary sub-\ca{} of $C$, it is nowhere scattered by \autoref{prp:permHerQuot}.
Thus, $B$ is nowhere scattered.
\end{proof}

%==========================================================================================
\begin{prp}
\label{prp:permSum}
Let $(A_j)_{j\in J}$ be a family of nowhere scattered \ca{s}.
Then the direct sum $\bigoplus_{j\in J} A_j$ is nowhere scattered.
\end{prp}
\begin{proof}
By \autoref{prp:permExt}, nowhere scatteredness passes to sums of finitely many summands.
Thus, for every finite subset $F\subseteq J$, the sum $\bigoplus_{j\in F} A_j$ is nowhere scattered.
Now the result follows from \autoref{prp:permLimit}, using that $\bigoplus_{j\in J} A_j$ is the inductive limit of $\bigoplus_{j\in F} A_j$, indexed over the finite subsets of $J$ ordered by inclusion.
\end{proof}

%==========================================================================================
The next example shows that nowhere scatteredness does not pass to products (of infinitely many summands).

%==========================================================================================
\begin{exa}
\label{exa:NoPermProd}
By \cite[Corollary~8.6]{RobRor13Divisibility}, there exists a sequence $(A_k)_{k\in\NN}$ of unital, simple, infinite-dimensional \ca{s} such that their product $\prod_k A_k$ has a one-dimensional, irreducible representation. 
Thus, while each $A_k$ is nowhere scattered, their product is not.

This example also shows that nowhere scatteredness does not pass to multiplier algebras of separable \ca{s}.
As an example, consider $A:=\bigoplus_k A_k$ with $A_k$ as above.
Then $A$ is separable and nowhere scattered by \autoref{prp:permSum}.
Further, it is well known that $M(A)$ is canonically isomorphic to $\prod_k A_k$; see, for example, \cite[II.8.1.3]{Bla06OpAlgs}.
\end{exa}

%==========================================================================================
\begin{exa}
\label{exa:NoPermMultSimple}
Let $M$ be a type~$\mathrm{II}_1$ factor, let $\varphi\colon M\to\CC$ be a pure state, and let $A$ be the associated hereditary kernel, that is,
\[
A = \big\{ a\in M : \varphi(aa^*)=\varphi(a^*a)=0 \big\}.
\]

By \cite[Theorem~1]{Sak71DerivationsSimple3}, $A$ is a simple \ca{} such that $D(A)/A\cong\CC$, where $D(A)$ is Sakai's derived algebra.
Pedersen showed that the derived algebra of a simple \ca{} is naturally isomorphic with its multiplier algebra;
see the remarks after Proposition~2.6 in \cite{Ped72ApplWeakSemicontinuity}.

Thus, $A$ is a simple, nowhere scattered \ca{} with $M(A)/A\cong\CC$, whence $M(A)$ is not nowhere scattered.
\end{exa}

%==========================================================================================
Examples~\ref{exa:NoPermProd} and \ref{exa:NoPermMultSimple} above show that for a nonunital, nowhere scattered \ca{} $A$ the multiplier algebra $M(A)$ may have a one-dimensional irreducible representation (and hence $M(A)$ is not nowhere scattered) even if we additionally assume that $A$ is separable or simple.
We suspect that there are also examples for the case that $A$ is separable \emph{and} simple.

%==========================================================================================
\begin{qst}
Does there exist a nonunital, separable, simple, nonelementary \ca{} $A$ such that $M(A)$ has a one-dimensional irreducible representation?
\end{qst}

%==========================================================================================
%==========================================================================================
\section{Topological characterizations}
\label{sec:topological}

%==========================================================================================
A topological space $X$ is said to be \emph{scattered} (or \emph{dispersed}) if every nonempty closed subset $C$ of $X$ contains a point that is isolated relative to $C$.

%==========================================================================================
\begin{dfn}
\label{dfn:spaceNWS}
We say that a topological space $X$ is \emph{nowhere scattered} if no closed subset of $X$ contains an isolated point.
\end{dfn}

%==========================================================================================
A subset of a topological space is said to be \emph{perfect} if it is closed and contains no isolated points.
Thus, a topological space is nowhere scattered if and only if each of its closed subsets is perfect.
It follows that nowhere scatteredness passes to closed subspaces.
Further, using that an isolated point in an open subset is also isolated in the whole space, we  see that nowhere scatteredness passes to open subspaces, and thus to locally closed subspaces.
Considering one-element subsets shows that nowhere scatteredness does not pass to every subspace.

%==========================================================================================
\begin{prp}
\label{prp:charSpaceNWS}
Let $X$ be a topological space.
Then the following are equivalent:
\begin{enumerate}
\item
$X$ is nowhere scattered; % (no closed subset contains an isolated point);
\item
$X$ has no nonempty, scattered, locally closed subsets;
\item
$X$ has no nonempty, locally closed subset that is $T_1$;
\item
$X$ has no locally closed subset containing only one element.
%\item
%every open subset of $X$ contains two topologically indistinguishable, distinct points.
%\item
%for every $k\geq 2$, every open subset of $X$ contains $k$ topologically indistinguishable, distinct points;
\end{enumerate}
\end{prp}
\begin{proof}
Since nowhere scatteredness passes to locally closed subsets, and since a nowhere scattered space is not scattered, we see that~(1) implies~(2).
Since a singleton is scattered, it follows that~(2) implies~(3).
It is clear that~(3) implies~(4).
Finally, it follows directly from the definition that~(4) implies~(1).
\end{proof}

%==========================================================================================
Let $A$ be a \ca{}.
We use $\widehat{A}$ to denote the spectrum of $A$, that is, the set of unitary equivalence classes of irreducible representations of $A$, equipped with the hull-kernel topology.
We refer to \cite[Paragraph~II.6.5.13]{Bla06OpAlgs} for details.

By \cite[Corollary~3]{Jen78ScatteredCAlg2}, $A$ is scattered if and only if $A$ is of type~$\mathrm{I}$ and $\widehat{A}$ is scattered as a topological space.
A separable \ca{} is of type~$\mathrm{I}$ if and only if its spectrum is a $T_0$-space;
see \cite[Theorem~IV.1.5.7]{Bla06OpAlgs}.
Since scattered spaces are $T_0$, it follows that a separable \ca{} $A$ is scattered if and only if $\widehat{A}$ is scattered.
The assumption of separability is necessary:
Akemann and Weaver's counterexample to the Naimark problem, \cite{AkeWea04CounterNaimark}, is a \ca{} that is nonelementary and simple (hence, not scattered), but whose spectrum is a one-point space (hence, scattered).
This also shows that the forward implication in \autoref{prp:charTopological} below does not hold for general nonseparable \ca{s}, although the backwards implication does.

%==========================================================================================
\begin{thm}
\label{prp:charTopological}
A separable \ca{} is nowhere scattered if and only if its spectrum is.
\end{thm}
\begin{proof}
Let $A$ be a separable \ca.
By \autoref{prp:firstChar}, $A$ is nowhere scattered if and only if it has no nonzero scattered ideal-quotients.
On the other hand, by \autoref{prp:charSpaceNWS}, $\widehat{A}$ is nowhere scattered if and only if it has no  nonempty scattered locally closed subsets. Using that a separable \ca{} is scattered if and only if its spectrum is (see \cite[Corollary~3]{Jen78ScatteredCAlg2}), the result now follows from the natural correspondence between ideal-quotients of $A$ and locally closed subsets of $\widehat{A}$.
\end{proof}

%==========================================================================================
%==========================================================================================
\section{Diffuse masas and Haar unitaries}
\label{sec:diffuse}

%==========================================================================================
In this section, we observe that a \ca{} is nowhere scattered if and only if each of its positive functionals is.
We use this result to connect nowhere scatteredness of a \ca{} to the existence of Haar unitaries and diffuse masas (maximal abelian sub-\ca{s}) for positive functionals.

Let $A$ be a \ca, and let $\varphi\colon A\to\CC$ be a positive functional.
We say that~$\varphi$ is \emph{nowhere scattered} if it gives no weight to scattered ideal-quotients of $A$, that is, if $\|\varphi|_I\|=\|\varphi|_J\|$ whenever $I\subseteq J\subseteq A$ are ideals such that $J/I$ is scattered;
see \cite[Definition~3.5]{Thi20arX:diffuseHaar}.

If $A$ is unital, then a unitary $u\in A$ is said to be a \emph{Haar unitary} with respect to~$\varphi$ if $\varphi(u^k)=0$ for all $k\in\mathbb{Z}\setminus\{0\}$;
see \cite[Definition~4.8]{Thi20arX:diffuseHaar}.
This definition is a generalization to the setting of positive functionals of the well-established notion of Haar unitaries with respect to traces.
By \cite[Proposition~4.9]{Thi20arX:diffuseHaar}, $\varphi$ admits a Haar unitary if and only if there exists a unital (maximal) abelian sub-\ca{} $C(X)\subseteq A$ such that $\varphi$ induces a diffuse measure on $X$. We extend this to the nonunital setting in \autoref{prp:nonunital} below.

%==========================================================================================
\begin{lma}
\label{prp:nonunital}
Let $\varphi$ be a positive functional on a nonunital \ca{} $A$, and let $\widetilde{A}:=A+\CC 1\subseteq A^{**}$ be the minimal unitization of $A$.
Set $\tilde{\varphi}:=\varphi^{**}|_{\widetilde{A}}\colon\widetilde{A}\to\CC$, which is the canonical extension of $\varphi$ to a positive functional on $\widetilde{A}$.
Then $\tilde{\varphi}$ admits a Haar unitary if and only if there exists a maximal abelian sub-\ca{} $C_0(X)\subseteq A$ such that $\varphi$ induces a diffuse measure $\mu$ on $X$ with $\mu(X)=\|\varphi\|$.
\end{lma}
\begin{proof}
To show the backward implication, assume that we have an abelian sub-\ca{} $C_0(X)\subseteq A$ with the stated properties.
We identify $C_0(X)+\CC 1\subseteq\widetilde{A}$ with $C(\tilde{X})$, where $\tilde{X}$ is the forced one-point compactification of $X$.
Recall that, if $X$ is already compact, then $\widetilde{X}$ is the disjoint union of $X$ and one extra point.

Let $\tilde{\mu}$ be the measure on $\widetilde{X}$ induced by $\tilde{\varphi}$.
Then
\[
\tilde{\mu}(\widetilde{X})
= \|\tilde{\varphi}\|
= \|\varphi\|
= \mu(X).
\]

Hence, $\tilde{\mu}(\widetilde{X}\setminus X)=0$, which implies that $\tilde{\mu}$ is diffuse.
It follows that a Haar unitary for $\tilde{\varphi}$ can be found in $C_0(X)+\CC 1\subseteq\widetilde{A}$.

Conversely, assume that $\tilde{\varphi}$ admits a Haar unitary.
By \cite[Proposition~4.9]{Thi20arX:diffuseHaar}, we obtain a maximal abelian sub-\ca{} $C(Y)\subseteq\widetilde{A}$ such that $\tilde{\varphi}$ induces a diffuse measure $\mu$ on $Y$.
We claim that $A\cap C(Y)$ has the desired properties.

Note that $1\in C(Y)$.
Let $\pi\colon\widetilde{A}\to\CC$ be the canonical one-dimensional representation such that $\ker(\pi)=A$.
The restriction of $\pi$ to $C(Y)$ corresponds to the evaluation at a point $y\in Y$, and the ideal $A\cap C(Y)$ of $C(Y)$ naturally corresponds to the open subset $X:=Y\setminus\{y\}$. 
We identify $C_0(X)$ with $A\cap C(Y)$, and we note that the measure on $X$ induced by $\varphi|_{C_0(X)}$ is the restriction of $\mu$ to $X$.
This measure on $X$ is therefore diffuse.

Using that $\mu$ is diffuse on $Y$, we have $\mu(X)=\mu(Y)$ and thus
\[
\mu(X)
= \mu(Y)
= \|\tilde{\varphi}\|
= \|\varphi\|.
\]

Finally, using that $C(Y)\subseteq\widetilde{A}$ is maximal abelian, it follows that $C_0(X)\subseteq A$ is maximal abelian as well.
\end{proof}

%==========================================================================================
\begin{thm}
\label{prp:charDiffuse}
Let $A$ be a \ca{}.
Then the following are equivalent:
\begin{enumerate}
\item
$A$ is nowhere scattered;
\item
every positive functional on $A$ is nowhere scattered;
\item
every pure state on $A$ is nowhere scattered;
\item
for every positive functional  $\varphi\colon A\to\CC$ and every hereditary sub-\ca{} $B\subseteq A$ there exists a Haar unitary in $\widetilde{B}$ with respect to $\varphi$;
\item
for every pure state $\varphi\colon A\to\CC$ and every ideal $I\subseteq A$ there exists a Haar unitary in $\widetilde{I}$ with respect to $\varphi$;
\item
for every positive functional $\varphi\colon A\to\CC$ and every hereditary sub-\ca{} $B\subseteq A$ there exists a maximal abelian sub-\ca{} $C_0(X)\subseteq B$ such that $\varphi$ induces a diffuse measure $\mu$ on $X$ with $\mu(X)=\|\varphi|_B\|$;
\item
for every pure state $\varphi\colon A\to\CC$ and every ideal $I\subseteq A$ there exists a maximal abelian sub-\ca{} $C_0(X)\subseteq I$ such that $\varphi$ induces a diffuse measure $\mu$ on $X$ with $\mu(X)=\|\varphi|_I\|$.
\end{enumerate}
\end{thm}
\begin{proof}
By \autoref{prp:firstChar}, $A$ is nowhere scattered if and only if $A$ has no nonzero scattered ideal-quotients.
This shows that~(1) implies~(2), which in turn implies~(3).
To see that~(3) implies~(1), assume for the sake of contradiction that $A$ is not nowhere scattered.
Then there exist ideals $I\subseteq J\subseteq A$ such that $J/I$ is nonzero and scattered.
Choose a nonzero pure state $\varphi_0$ on $J/I$.
Composing with the quotient map $J\to J/I$ we obtain a pure state $\varphi$ on $J$, which we can extend to a pure state~$\tilde{\varphi}$ on~$A$.
Then $\tilde{\varphi}|_I=0$, while $\tilde{\varphi}|_J\neq 0$, which shows that $\tilde{\varphi}$ is a pure state on $A$ that is \emph{not} nowhere scattered, a contradiction.

By \cite[Theorem~4.11]{Thi20arX:diffuseHaar}, (2) is equivalent to~(4).
Similarly, (3) is equivalent to~(5).
Finally, it follows from \autoref{prp:nonunital} that~(4) is equivalent to~(6), and that~(5) is equivalent to~(7).
\end{proof}

%==========================================================================================
%==========================================================================================
\section{A new property of Cuntz semigroups}
\label{sec:O8}

%==========================================================================================
We introduce a new property for \CuSgp{s}, called \axiomO{8}; see \autoref{dfn:O8}. This new property can be thought of as the \CuSgp{} version of the projectivity of $C_0((0,1])\oplus C_0((0,1])$, which allows one to lift orthogonal positive elements in a quotient of a \ca.
We show the the Cuntz semigroup of every \ca{} satisfies \axiomO{8};
see \autoref{prp:CuHasO8}.
We also deduce a result that can be interpreted as the $\CatCu$-version of the projectivity of $M_n(C_0((0,1]))$;
see \autoref{prp:O8DivRef}.

%==========================================================================================
\begin{pgr}
\label{pgr:Cu}
Let $A$ be a \ca.
Given $a,b\in A_+$, one says that $a$ is \emph{Cuntz subequivalent} to~$b$, denoted $a\precsim b$, if there is a sequence $(r_n)_n$ in $A$ such that $\lim_n \|a-r_nbr_n^*\|=0$.
Further, $a$ is said to be \emph{Cuntz equivalent} to $b$, in symbols $a\sim b$, if $a\precsim b$ and $b\precsim a$.
The \emph{Cuntz semigroup} of $A$ is the set of Cuntz equivalence classes of positive elements in the stabilization of $A$, that is,
\[
\Cu(A) := (A\otimes\mathcal{K})_+/\!\sim,
\]
where the class of a positive element $a$ is denoted by $[a]$.
One equips $\Cu(A)$ with the  addition induced by the orthogonal sum, and with the partial order given by $[a]\leq[b]$ if $a\precsim b$.
This turns $\Cu(A)$ into a partially ordered, abelian monoid.

As shown in \cite{CowEllIva08CuInv}, the Cuntz semigroup of any \ca{} enjoys additional order-theoretic properties; see also \cite{AraPerTom11Cu} and \cite{AntPerThi18TensorProdCu}. To state them, recall that given two elements $x,y$ in a partially ordered set, one says that $x$ is \emph{way-below}~$y$, denoted $x\ll y$, if for every increasing sequence $(y_k)_k$ with supremum satisfying $y\leq\sup_k y_k$ there exists $k'$ such that $x\leq y_{k'}$.
Given a \ca{} $A$, $[a]\ll[b]$ in $\Cu (A)$ if and only if there exists $\varepsilon>0$ such that $a\precsim (b-\varepsilon)_+$, where $(b-\varepsilon)_+$ is the \emph{$\varepsilon$-cut-down} of~$b$ obtained by applying functional calculus for the function $t\mapsto\max\{0,t-\varepsilon\}$ to~$b$.

One says that a positively ordered, abelian monoid $S$ is a \emph{\CuSgp{}}, also called \emph{abstract Cuntz semigroup}, if the following conditions are satisfied:
\begin{enumerate}
\item[\axiomO{1}]
Every increasing sequence in $S$ has a supremum.
\item[\axiomO{2}]
For every $x\in S$ there exists a $\ll$-increasing sequence $(x_n)_n$ with $x=\sup_n x_n$.
\item[\axiomO{3}]
If $x'\ll x$ and $y'\ll y$ in $S$, then $x'+y'\ll x+y$.
\item[\axiomO{4}]
If $(x_n)_n$ and $(y_n)_n$ are increasing sequences in $S$, then $\sup_n(x_n+y_n)=\sup_nx_n + \sup_ny_n$.
\end{enumerate}
By \cite{CowEllIva08CuInv}, $\Cu(A)$ is a \CuSgp{} for every \ca{} $A$.

Since the introduction of \CuSgp{s}, the Cuntz semigroup of any \ca{} has been shown to satisfy additional properties, such as the two stated below; see \cite{AntPerThi18TensorProdCu} and \cite{Rob13Cone} respectively, and also \cite{AntPerRobThi21Edwards}.
\begin{enumerate}
\item[\axiomO{5}]
Given $x+y\leq z$, $x'\ll x$ and $y'\ll y$ in $S$ there exists $c\in S$ such that $x'+c\leq z\leq x+c$ and $y'\ll c$.
(This property is often applied with $y'=y=0$, in which case it says that for $x'\ll x\leq z$ there exists an `almost complement' $c$ such that $x'+c\leq z\leq x+c$.)
\item[\axiomO{6}]
Given $x'\ll x\leq y+z$ in $S$ there exist $e,f\in S$ such that $x'\leq e+f$ with $e\leq x,y$, and $f\leq x,z$.
\end{enumerate}
\end{pgr}

%==========================================================================================

\begin{dfn}
\label{dfn:O8}
Let $S$ be a \CuSgp{}.
We say that $S$ satisfies \axiomO{8} if for all $x',x,y',y,z,w\in S$ satisfying $2w=w$ and 
\[
x+y \ll z+w, \quad
x'\ll x, \andSep
y'\ll y
\]
there exist $z_1,z_2\in S$ such that
\[
z_1+z_2\ll z, \quad
x'\ll z_1+w, \quad
y'\ll z_2+w, \quad
z_1\ll x+w, \andSep
z_2\ll y+w.
\]
\end{dfn}

%==========================================================================================
\begin{rmk}
\label{rmk:O8Reduc}
A \CuSgp{} $S$ satisfies \axiomO{8} if and only if for all $x',x,y',y,z,w$ in $S$ satisfying $2w=w$ and 
\[
x+y \ll z+w, \quad
x'\ll x, \andSep
y'\ll y
\]
there exist $z_1,z_2\in S$ such that
\[
z_1+z_2\leq z, \quad
x'\leq z_1+w, \quad
y'\leq z_2+w, \quad
z_1\leq x+w, \andSep
z_2\leq y+w.
\]

Indeed, the forward implication is trivial. To see the backward implication, let $x',x,y',y,z,w$ satisfy the conditions in the statement, and choose $\tilde{x},\tilde{y},\tilde{z}$  such that
\[
 x+y\ll \tilde{z}+w,\quad 
 x'\ll \tilde{x}\ll x,\quad
 y'\ll \tilde{y}\ll y,\andSep 
 \tilde{z}\ll z.
\]

Then, there exist $t_1,t_2\in S$ with
\[
t_1+t_2\leq \tilde{z}, \quad
\tilde{x}\leq t_1+w, \quad
\tilde{y}\leq t_2+w, \quad
t_1\leq x+w, \andSep
t_2\leq y+w.
\]

In particular, since $x'\ll t_1+w$ and $y'\ll t_2+w$, we can find elements $z_1\ll t_1$ and $z_2\ll t_2$ such that $x'\ll z_1+w$ and $y'\ll z_2+w$. It is easy to check that such elements satisfy the desired conditions.
\end{rmk}

%==========================================================================================
\begin{thm}
\label{prp:CuHasO8}
The Cuntz semigroup of every \ca{} satisfies \axiomO{8}.
\end{thm}
\begin{proof}
Let $A$ be a \ca, and let $x',x,y',y,z,w\in\Cu(A)$ satisfy $2w=w$ and 
\[
x+y \ll z+w, \quad
x'\ll x, \andSep
y'\ll y.
\]
%Choose $z'\in\Cu(A)$ such that
%\[
%x+y\ll z'+w, \andSep z'\ll z.
%\]

We may assume that $A$ is stable.
Let $I$ be the ideal of $A$ corresponding to the ideal $\{s\in\Cu(A):s\leq w\}$ of $\Cu(A)$, and let $\pi\colon A\to A/I$ denote its quotient map.
Choose $a,b,c\in A_+$ and $\varepsilon>0$ such that $a$ and $b$ are orthogonal,
\[
x=[a], \quad
y=[b], \quad
z=[c], \quad
x'\leq[(a-\varepsilon)_+], \andSep
y'\leq[(b-\varepsilon)_+].
\]

By \cite[Proposition~3.3]{CiuRobSan10CuIdealsQuot}, two elements $e,f\in A_+$ satisfy $[e]\leq [f]+w$ if and only if $\pi(e)\precsim\pi(f)$.
Thus, 
\[
\pi (a)+\pi (b)= \pi(a+ b) \precsim \pi(c).
\]

Using R{\o}rdam's Lemma (see for example \cite[Theorem~2.30]{Thi17:CuLectureNotes}), there exists $r\in A/I$ such that
\[
(\pi(a)-\varepsilon)_+ + (\pi(b)-\varepsilon)_+
= ((\pi(a) + \pi(b))-\varepsilon)_+
= r^*r
\]
and $rr^* \in \overline{\pi(c)(A/I)\pi(c)}$.

Set
\[
e := r(\pi(a)-\varepsilon)_+r^*, \andSep
f := r(\pi(b)-\varepsilon)_+r^*.
\]

Note that $e$ and $f$ are orthogonal positive elements contained in the hereditary sub-\ca{} $\overline{\pi(c)(A/I)\pi(c)}$.
Since $\pi$ maps the hereditary sub-\ca{} $\overline{cAc}$ onto $\overline{\pi(c)(A/I)\pi(c)}$, and since the \ca{} $C_0((0,\|e\|])\oplus C_0((0,\|f\|])$ is projective (see \cite[Section~4]{EffKam86ShapeThy}), we can choose orthogonal positive elements $\tilde{e},\tilde{f}\in\overline{cAc}$ such that $\pi(\tilde{e})=e$ and $\pi(\tilde{f})=f$.
Set
\[
z_1 := [\tilde{e}], \andSep
z_2 := [\tilde{f}].
\]

Using that $\tilde{e}$ and $\tilde{f}$ are orthogonal, and that $\tilde{e}+\tilde{f}\in\overline{cAc}$, we get $z_1+z_2\leq[c]=z$.

Also note that, since $\pi(\tilde{e})=e\precsim (\pi (a)-\epsilon )_+\precsim \pi (a)$, we have $z_1\leq x+w$. Similarly, $z_2\leq y+w$.

Moreover, one also gets
\[
 r^* er= r^*r(\pi(a)-\varepsilon)_+r^*r
 =(\pi(a)-\varepsilon)^3_+\sim (\pi(a)-\varepsilon)_+,
\]
which shows that $x'\leq z_1+w$. An analogous argument proves $y'\leq z_2+w$.

It follows from \autoref{rmk:O8Reduc} that $\Cu (A)$ satisfies \axiomO{8}, as desired.
\end{proof}

%==========================================================================================
One can think of \axiomO{8} as a weak form of Riesz refinement. 
In this sense, \autoref{prp:WCO5O6ImpO8} below can be seen as a $\CatCu$-version of the fact that a cancellative, algebraically ordered semigroup with Riesz decomposition has Riesz refinement.

A \CuSgp{} $S$ is said to be \emph{weakly cancellative} if for all $x,y,z\in S$ with $x+z\ll y+z$ we have $x\ll y$.
It follows from \cite[Theorem~4.3]{RorWin10ZRevisited} that the Cuntz semigroup of every stable rank one \ca{} is weakly cancellative.

%==========================================================================================
\begin{prp}
\label{prp:WCO5O6ImpO8}
Let $S$ be a weakly cancellative \CuSgp{} satisfying \axiomO{5} and \axiomO{6}.
Then $S$ satisfies \axiomO{8}.
\end{prp}
\begin{proof}
Let $x',x,y',y,z,w$ satisfy
\[
x+y \ll z+w, \quad
x'\ll x, \quad
y'\ll y, \andSep
2w=w.
\]

Applying \axiomO{6} for $x'\ll x\ll z+w$, we obtain $\tilde{z}_1$ such that
\[
    x' \ll \tilde{z}_1+w, \andSep \tilde{z}_1\ll x,z.
\]

Choose $z_1\in S$ such that
\[
    x'\ll z_1+w, \andSep z_1\ll \tilde{z}_1.
\]

Applying \axiomO{5} for $z_1\ll \tilde{z}_1\leq z$, we find $c\in S$ with
\[
z_1+c\leq z\leq \tilde{z}_1+c.
\]

Then,
\[
x+y \ll z+w 
\leq \tilde{z}_1+c+w
\]
and, since $S$ is weakly cancellative and $\tilde{z}_1\leq x$, we get $y\ll c+w$.

Applying \axiomO{6} for $y'\ll y\ll c+w$, we obtain $z_2$ such that
\[
    y' \ll z_2+w, \andSep z_2\ll y,c.
\]

We have $z_1+z_2\leq z_1+c\leq z$. It now follows from \autoref{rmk:O8Reduc} that $z_1$ and $z_2$ have the desired properties.
\end{proof}

%==========================================================================================
\begin{lma}
\label{prp:O8Multiple}
Let $S$ be a \CuSgp{} satisfying \axiomO{8}, let $w\in S$ satisfy $2w=w$, and let $x_1',\ldots,x_n',x_1,\ldots,x_n,z\in S$ satisfy
\[
x_1+\ldots+x_n \ll z+w, \quad
x_1'\ll x_1, \quad \ldots, \andSep
x_n'\ll x_n.
\]

Then there exist $z_1,\ldots,z_n\in S$ such that
\[
z_1+\ldots+z_n\ll z,\quad 
x_j'\ll z_j+w,\andSep 
z_j\ll x_j+w
\]
for $j=1,\ldots,n$.
\end{lma}
\begin{proof}
We prove the result by induction over $n$.
The case $n=1$ is clear, and the case $n=2$ holds by definition of \axiomO{8}.

Thus, let $n\geq 3$, and assume that the statement holds for $n-1$.
Let $x_j',x_j,z,w$ for $j=1,\ldots,n$ be as in the statement, and choose $x_{n-1}'',x_n''$ such that
\[
x_{n-1}'\ll x_{n-1}''\ll x_{n-1}, \andSep
x_n' \ll x_n'' \ll x_n.
\]

Applying the assumption for the $n-1$ pairs $x_1'\ll x_1,\ldots,x_{n-2}'\ll x_{n-2}$ and $x_{n-1}''+x_n''\ll x_{n-1}+x_n$, we obtain $z_1,\ldots,z_{n-2},v$ such that
\[
z_1+\ldots+z_{n-2}+v\ll z, \quad
x_j'\ll z_j+w, \andSep
z_j\ll x_j+w
\]
for $j=1,\ldots,n-2$ and 
\[
x_{n-1}''+x_n''\ll v+w, \andSep
v\ll x_{n-1}+x_n+w.
\]

Applying \axiomO{8} for
\[
x_{n-1}''+x_n''\ll v+w, \quad
x_{n-1}'\ll x_{n-1}'', \andSep
x_n'\ll x_n'',
\]
we obtain $z_{n-1},z_n$ with $z_{n-1}+z_n\ll v$ and such that
\[
x_{n-1}'\ll z_{n-1}+w, \quad
x_n'\ll z_n+w, \quad
z_{n-1}\ll x_{n-1}''+w, \andSep
z_n\ll x_n''+w.
\]

Then $z_1,\ldots,z_n$ have the desired properties.
\end{proof}

%==========================================================================================
\begin{prp}
\label{prp:O8MultipleOrdered}
Let $S$ be a \CuSgp{} satisfying \axiomO{6} and \axiomO{8}, let $w\in S$ be such that $2w=w$, and let $x_1',\ldots,x_n',x_1,\ldots,x_n,z\in S$ satisfy
\[
x_1+\ldots+x_n \ll z+w, \andSep
x_1'\ll x_1 \ll x_2' \ll x_2 \ll \ldots \ll x_n'\ll x_n.
\]

Then, there exist $z_1,\ldots,z_n\in S$ such that
\[
z_1+z_2+\ldots+z_n\ll z, \quad 
z_1\ll\ldots\ll z_n,\quad 
x_j'\ll z_j+w,\andSep 
z_j\ll x_j+w
\]
for $j=1,\ldots,n$.
\end{prp}
\begin{proof}
Applying \autoref{prp:O8Multiple}, we obtain $y_1,\ldots,y_n$ such that
\[
y_1+\ldots+y_n\ll z, \quad
x_j'\ll y_j+w, \andSep
y_j\ll x_j+w
\]
for $j=1,\ldots,n$.

For every $j$, let $y_j'\ll y_j$ be such that $x_j'\ll y_j'+w$. Set $z_n:=y_n$, and note that
\[
y_{n-1}'\ll y_{n-1} \leq x_{n-1}+w \leq x_n'+w \leq y_n+w = z_n+w.
\]

Applying \axiomO{6}, we obtain $z_{n-1}$ such that
\[
y_{n-1}'\ll z_{n-1}+w, \andSep z_{n-1}\ll y_{n-1},z_n,
\]
where note that one also has 
\[
 y_{n-2}'\ll y_{n-2}\leq x_{n-2}+w\leq x_{n-1}'+w\leq y_{n-1}'+w\leq z_{n-1}+w.
\]

Proceeding in this manner, we obtain elements $z_1,\ldots ,z_n$ such that
\[
 y_{j}'\ll z_{j}+w, \andSep z_{j}\ll y_{j},z_{j+1}
\]
for every $j\leq n-1$.

It is easy to see that such elements satisfy the required properties.
\end{proof}

%==========================================================================================
The next result can be interpreted as the \CuSgp{} version of the projectivity of $C_0((0,1],M_n)$.

%==========================================================================================
\begin{prp}
\label{prp:O8DivRef}
Let $S$ be a \CuSgp{} satisfying \axiomO{8}, let $n\geq 1$, and let $x',x,y,w\in S$ satisfy
\[
nx\ll y+w, \quad 
x'\ll x, \andSep
2w=w.
\]

Then there exists $z\in S$ such that
\[
nz \ll y, \quad
x' \ll z + w, \andSep
z \ll x+w.
\]
\end{prp}
\begin{proof}
Choose $x_1',x_1,x_2',x_2,\ldots,x_n$ such that
\[
x'\ll x_1'\ll x_1\ll x_2'\ll x_2 \ll \ldots\ll x_n' \ll x_n \ll x.
\]

Applying \autoref{prp:O8MultipleOrdered}, we obtain $z_1,\ldots,z_n$ satisfying
\[
z_1+\ldots+z_n\ll y, \quad
z_1\ll \ldots \ll z_n, \quad
x_j'\ll z_j+w, \andSep
z_j\ll x_j+w
\]
for $j=1,\ldots,n$.

Set $z:=z_1$.
Then
\[
nz \leq z_1+\ldots+z_n \ll y, \quad
x' \leq x_1' \ll z_1+w = z+w,
\]
and $z = z_1 \leq z_n \ll x_n+w\leq x+w$.
\end{proof}

%==========================================================================================
%==========================================================================================
\section{C*-algebras and Cu-semigroups without elementary ideal-quotients}
\label{sec:EleIdeQuo}

%==========================================================================================
In this section we prove that a \CuSgp{} $S$ satisfying \axiomO{5}, \axiomO{6} and \axiomO{8} has no nonzero elementary ideal-quotients if and only if $S$ is weakly $(2,\omega )$-divisible; 
see \autoref{prp:IdealQuotient}. 
In \autoref{pgr:Elementary} we provide a tailored definition of \emph{elementary \CuSgp{}}, which is justified by \autoref{prp:CharElementary}. 
We deduce in \autoref{prp:charDiv} that a \ca{} is nowhere scattered if and only if every element in its Cuntz semigroup is weakly $(2,\omega)$-divisible.
This can also be deduced from results in \cite{RobRor13Divisibility};
see \autoref{rmk:RobRor}.

%==========================================================================================
\begin{pgr}\label{pgr:Elementary}
A \ca{} is said to be \emph{elementary} if it is isomorphic to the compact operators on some Hilbert space.

In \cite[Paragraph~5.1.16]{AntPerThi18TensorProdCu}, a nonzero \CuSgp{} was said to be `elementary' if it is simple and contains a minimal nonzero element.
This definition includes $\{0,\infty\}$, which is the Cuntz semigroup of simple, purely infinite \ca{s} -- and these \ca{s} are very far from elementary.
To amend this, we will instead say that a nonzero \CuSgp{} is \emph{elementary} if it is simple and contains a minimal, nonzero element $x$ that is \emph{finite} (that is, $x\neq 2x$).
\autoref{prp:CharElementary} below shows that this refined definition fits the established terminology in \ca{s}.
\end{pgr}

%==========================================================================================
The next result is shown in \cite[Theorem~4.4.4]{Eng14PhD}.
We include a proof for the convenience of the reader.

%==========================================================================================
\begin{lma}
\label{prp:CharElementary}
Let $A$ be a (nonzero) \ca.
Then the following are equivalent:
\begin{enumerate}
\item
$A$ is elementary;
\item
$\Cu(A)\cong\overline{\NN}$;
\item
$\Cu(A)$ is elementary.
\end{enumerate}
\end{lma}
\begin{proof}
To show that~(1) implies~(2), assume that $A$ is elementary.
Upon stabilization, we may assume that $A\cong\mathcal{K}(H)$ for some infinite-dimensional Hilbert space~$H$.
Using the spectral theorem for compact operators, one can show that $A_+\to\overline{\NN}$, mapping $a\in A_+$ to its rank, induces the desired isomorphism $\Cu(A)\cong\overline{\NN}$.

Alternatively, one sees that every projection in $\mathcal{K}(H)$ has finite-rank, and that two projections are Murray-von Neumann equivalent if and only their ranks agree.
It follows that the Murray-von Neumann semigroup $V(A)$ is isomorphic to $\NN$.
Using that $A$ has real rank zero, we obtain that $\Cu(A)$ is isomorphic to the sequential ideal completion of $\NN$, and thus $\Cu(A)\cong\overline{\NN}$;
see for example \cite[Remark 5.5.6]{AntPerThi18TensorProdCu}.

It is clear that~(2) implies~(3).
To show that~(3) implies~(1), assume that $\Cu(A)$ is elementary.
Then $\Cu(A)$ is simple, and consequently so is $A$ (see \cite[Corollary~5.1.12]{AntPerThi18TensorProdCu}).
Choose a minimal, nonzero element $x\in\Cu(A)$ with $x\neq 2x$, and let $a\in A_+$ be a  nonzero element.

Since $x$ is minimal, we have $x\ll x$.
Using that $A$ is simple, it follows from  \cite[Theorem~5.8]{BroCiu09IsoHilbModSF} that we can choose a projection $p\in A\otimes\mathcal{K}$ with $[p]=x$.
Since $a\neq 0$, we have $[a]\neq 0$ and therefore $x\leq[a]$.
Then $p\precsim a$, whence we obtain a projection $q\in A$ with $[q]=x$.
Thus, $q$ is a minimal open projection, that is, $qAq=\CC q$.
This is well-known to imply that $A$ is elementary.
\end{proof}

%==========================================================================================
Given a \CuSgp{} $S$, a downward-hereditary submonoid $I\subseteq S$ is an \emph{ideal} if it is closed under suprema of increasing sequences. For each ideal $I$ of $S$, one can consider its corresponding quotient $S/I$, which is again a \CuSgp{}; see \cite[Lemma 5.1.2]{AntPerThi18TensorProdCu}.

We omit the proof of the following lemma.

%==========================================================================================
\begin{lma}
\label{prp:O8Quot}
Let $S$ be a \CuSgp{} satisfying \axiomO{8} and let $I$ be an ideal of $S$. 
Then, $I$ and $S/I$ satisfy \axiomO{8}.
\end{lma}

%==========================================================================================
\begin{lma}
\label{prp:LiftDivO8}
Let $S$ be a \CuSgp{} satisfying \axiomO{5} and \axiomO{8}. 
Let $I$ be an ideal of $S$ and denote by $\pi$ the canonical map $S\to S/I$. 
Given $y\in S$ and $e,e'\in S/I$ such that 
\[
e'\ll e,\andSep 
2e\leq \pi (y),
\] 
there exists $z\in S$ such that $e'\ll \pi (z)\ll e$ and $2z\ll y$.
\end{lma}
\begin{proof}
Let $x\in S$ be such that $\pi (x)=e$ and take $x'\ll x$ such that $e'\ll \pi (x')$. 
Since $\pi (2x)\leq \pi (y)$, there exists $w\in I$ satisfying $2x\leq y+w$. 
Thus, we have $2x\leq y+\infty w$, where we note that $\infty w\in I$.
 
Applying \autoref{prp:O8DivRef}, we obtain an element $z\in S$ such that 
 \[
2z\ll y, \quad 
x'\ll z+\infty w, \andSep
z \ll x+\infty w.
\]

Passing to the quotient, and using that $\infty w\in I$, we get 
\[
 e'\ll \pi (x')\ll \pi (z)\ll \pi (x)=e,
\]
as required.
\end{proof}

%==========================================================================================
By an \emph{ideal-quotient} in a \CuSgp{} $S$ we mean a quotient $J/I$ for some ideals $I\subseteq J$ of $S$.
%We say that $S$ has \emph{no elementary ideal-quotients} if there are no ideals $I\subseteq J\subseteq S$ such that $J/I$ is \emph{nonzero} and elementary.

%==========================================================================================
\begin{lma}
\label{prp:GlimmO8}
Let $S$ be a \CuSgp{} satisfying \axiomO{5}, \axiomO{6} and \axiomO{8}.
Assume that $S$ has no nonzero elementary ideal-quotients. 
Let $x\in S$ be nonzero.
Then there exists $z\in S$ with $0\neq 2z\leq x$.
\end{lma}
\begin{proof}
Using that $x$ is nonzero, we can choose nonzero elements $x'',x'\in S$ such that $x''\ll x'\ll x$.

Applying \axiomO{5}, we obtain $c\in S$ with
\[
x''+c\leq x\leq x'+c.
\]

Let $I\subseteq S$ be the ideal generated by $c$. Then, $x$ is either in $I$ or not. We study each case separately.

\textbf{Case~1:} We have $x\in I$.
In this case, and since ideals are downward hereditary, we have $x''\in I$, and therefore $x''\ll \infty c$. Thus, there exists $n\in \NN$ such that $x''\leq nc$.

Let $x'''\ll x''$ be a nonzero element. 
Following the proof of \cite[Proposition~5.2.1]{Rob13Cone}, we can apply \axiomO{6} to $x'''\ll x''\leq nc$ to obtain elements $c_1,\ldots ,c_n\leq x'',c$ such that
\[
 x'''\leq c_1+\ldots +c_n.
\]

Using that $x'''\neq 0$, it follows that there exists some $j$ such that the element $c_j$ is nonzero. Setting $z=c_j$, one has $2z\leq x''+c\leq x$, as required.

\textbf{Case~2:} We have $x\notin I$. 
Let $K\subseteq S$ denote the ideal generated by $x$.
Then the image of $x$ in $K/I$ is a nonzero, compact, full element.
%Note that, in this case, we may assume $S=(x)$.

It follows that there exists a maximal ideal $J\subseteq K$ containing $I$.
Consequently, the quotient $K/J$ is simple.
Let $\pi\colon K\to K/J$ be the quotient map. 

By \cite[Proposition~5.1.3]{AntPerThi18TensorProdCu}, \axiomO{5} and \axiomO{6} pass to ideals and quotients.
Thus, $K/J$ is a simple, nonelementary \CuSgp{} satisfying \axiomO{5} and \axiomO{6}.
Applying \cite[Proposition~5.2.1]{Rob13Cone}, we obtain $e\in K/J$ with $0\neq 2e\leq\pi(x)$.
Choose $e'\in K/J$ with $0\neq e'\ll e$.
By Lemmas \ref{prp:O8Quot} and \ref{prp:LiftDivO8}, we obtain $z\in K\subseteq S$ such that $e'\ll\pi(z)$ and $2z\leq x$.
It follows that $z$ is nonzero and has the desired properties.
\end{proof}

%==========================================================================================
\begin{pgr}
Let $S$ be a \CuSgp{} and let $k\geq 1$. An element $x\in S$ is said to be \emph{$(k,\omega)$-divisible} if for every $x'\in S$ satisfying $x'\ll x$ there exist $n\in\NN$ and $y\in S$ such that $ky\leq x$ and $x'\leq ny$.
Further, $x$ is said to be \emph{weakly $(k,\omega)$-divisible} if for every $x'\in S$ satisfying $x'\ll x$ there exist $n\in\NN$ and $y_1,\ldots,y_n\in\Cu(A)$ such that $ky_1,\ldots,ky_n\leq x$ and $x'\leq y_1+\ldots+y_n$;
see \cite[Definition~5.1]{RobRor13Divisibility} or \cite[Paragraph~5.1]{AntPerRobThi22CuntzSR1}.

We say that $S$ is (weakly) $(k,\omega)$-divisible if each of its elements is.
%Let $A$ be a \ca{}, let $x\in\Cu(A)$, and let $k\geq 1$.
%Then $x$ is said to be \emph{$(k,\omega)$-divisible} if for every $x'\in\Cu(A)$ satisfying $x'\ll x$ there exist $n\in\NN$ and $y\in\Cu(A)$ such that $ky\leq x$ and $x'\leq ny$.
%Further, $x$ is said to be \emph{weakly $(k,\omega)$-divisible} if for every $x'\in\Cu(A)$ satisfying $x'\ll x$ there exist $n\in\NN$ and $y_1,\ldots,y_n\in\Cu(A)$ such that $ky_1,\ldots,ky_n\leq x$ and $x'\leq y_1+\ldots+y_n$.
%See \cite[Definition~5.1]{RobRor13Divisibility} or \cite[Paragraph~5.1]{AntPerRobThi22CuntzSR1}.
\end{pgr}

%==========================================================================================
\begin{lma}
\label{prp:2DivImplKDiv}
Let $S$ be a weakly $(2,\omega)$-divisible \CuSgp{}.
Then, $S$ is weakly $(k,\omega)$-divisible for every $k\in\NN$.
\end{lma}
\begin{proof}
It suffices to show that every element in $S$ is weakly $(2^k,\omega)$-divisible for each $k\geq 1$.
We prove this by induction over $k$, where we note that the case $k=1$ holds by assumption.

Thus, let $k\geq 1$ and assume that every element is weakly $(2^k,\omega)$-divisible.
To verify that every element is weakly $(2^{k+1},\omega)$-divisible, let $x',x\in S$ satisfy $x'\ll x$.
Choose $x''\in S$ such that $x'\ll x''\ll x$.
Using that $x$ is weakly $(2,\omega)$-divisible, we obtain $m\in\NN$ and $y_1,\ldots,y_m\in S$ such that
\[
2y_1,\ldots,2y_m \leq x, \andSep x''\leq y_1+\ldots+y_m.
\]

Choose $y_1',\ldots,y_m'\in S$ such that
\[
y_1'\ll y_1, \quad\ldots, \quad
y_m'\ll y_m, \andSep
x'\leq y_1'+\ldots+y_m'.
\]

Applying the induction assumption for each pair $y_i'\ll y_i$, we obtain $n(i)\in\NN$ and $z_{i,1},\ldots,z_{i,n(i)}\in S$ such that
\[
2^k z_{i,1}, \ldots, 2^k z_{i,n(i)} \leq y_i, \andSep y_i'\leq z_{i,1}+\ldots+z_{i,n(i)}. 
\]

Consequently, 
\[
2^{k+1} z_{i,j} \leq 2y_i \leq x,
\]
for each $i\in\{1,\ldots,m\}$ and $j\in\{1,\ldots,n(i)\}$.
Further, 
\[
x' 
\leq y_1'+\ldots+y_m'
= \sum_{i=1}^m y_i'
\leq \sum_{i=1}^m \sum_{j=1}^{n(i)} z_{i,j},
\]
which shows that the elements $z_{i,j}$ have the desired properties.
\end{proof}

%==========================================================================================
\begin{prp}
\label{prp:IdealQuotient}
Let $S$ be a \CuSgp{} satisfying \axiomO{5}, \axiomO{6} and \axiomO{8}.
Then the following are equivalent:
\begin{enumerate}
\item
$S$ has no nonzero elementary ideal-quotients;
\item
every element in $S$ is weakly $(2,\omega)$-divisible;
\item
every element in $S$ is weakly $(k,\omega)$-divisible for every $k\geq 2$.
\end{enumerate}
\end{prp}
\begin{proof}
The equivalence between~(2) and~(3) follows from \autoref{prp:2DivImplKDiv}.
To show that~(2) implies~(1), assume that $S$ is weakly $(2,\omega)$-divisible, and note that a nonzero elementary \CuSgp{} $T$ is not weakly $(2,\omega)$-divisible.
Indeed, if $x\in T$ is minimal, nonzero and $x\neq 2x$, then one can see that $x$ is not weakly $(2,\omega )$-divisible.
Since weak $(2,\omega)$-divisibility passes to ideals and quotients, it follows that $S$ has no nonzero elementary ideal-quotients.
 
Conversely, to show that~(1) implies~(2), assume now that $S$ has no nonzero elementary ideal-quotients and let $x\in S$. 
The proof is inspired by that of \cite[Theorem~6.7]{AraGooPerSil10NonSimplePI}.
Set
\[
D:= \big\{ z\in S : 2z\leq x \big\}.
\]

Let $I\subseteq S$ be the ideal generated by $D$.
Note that $s\in S$ belongs to $I$ if and only if there exists a sequence $(z_k)_k$ in $D$ such that $s\leq\sum_{k=0}^\infty z_k$.

\textbf{Claim:} \emph{We have $x\in I$.}

To prove the claim, let $\pi\colon S\to S/I$ denote the quotient map.
To reach a contradiction, assume that $\pi(x)\neq 0$. 
By \cite[Proposition~5.1.3]{AntPerThi18TensorProdCu} and  \autoref{prp:O8Quot}, we know that \axiomO{5}, \axiomO{6} and \axiomO{8} pass to quotients.

Hence, we can apply \autoref{prp:GlimmO8} to $S/I$.
Thus, since $\pi(x)\neq 0$, we obtain $e\in S/I$ with $0\neq 2e\leq \pi (x)$.
Choose $e'\in S/I$ with $0\neq e'\ll e$.
By \autoref{prp:LiftDivO8}, we obtain $z\in S$ such that $e'\ll\pi(z)$ and $2z\leq x$.
Consequently, $z\in D\subseteq I$, which implies $\pi(z)=0$ and therefore $e'=0$, a contradiction.
Thus, we have $\pi(x)=0$ and so $x\in I$, which proves the claim.

Now, given any $x'\in S$ such that $x'\ll x$,  take $u\in S$ with $x'\ll u\ll x$.
Using that $u\ll x\in I$, we obtain $z_1,\ldots,z_n\in D$ such that
$u \leq z_1+\ldots+z_n $.

Note that we have $2z_j\leq x$ for every $j$ and $x'\ll u \leq z_1+\ldots +z_n$. It follows that $x$ is weakly $(2,\omega )$-divisible, as desired.
\end{proof}

%==========================================================================================
\begin{thm}
\label{prp:charDiv}
Let $A$ be a \ca.
Then the following are equivalent:
\begin{enumerate}
\item
$A$ is nowhere scattered;
\item
every element in $\Cu(A)$ is weakly $(2,\omega)$-divisible;
\item
every element in $\Cu(A)$ is weakly $(k,\omega)$-divisible for every $k\geq 2$;
\end{enumerate}
\end{thm}
\begin{proof}
It follows from \cite[Proposition~5.1.10]{AntPerThi18TensorProdCu} that ideal-quotients in $A$ naturally correspond to ideals-quotients in $\Cu(A)$.
Thus, by \autoref{prp:CharElementary}, $A$ has no nonzero elementary ideal-quotients if and only if $\Cu(A)$ has no nonzero elementary ideal-quotients.
Applying \autoref{prp:firstChar}, we see that $A$ is nowhere scattered if and only if $\Cu(A)$ has no nonzero elementary ideal-quotients.

As noted in \autoref{pgr:Cu}, $\Cu(A)$ is a \CuSgp{} satisfying \axiomO{5} and \axiomO{6}.
By \autoref{prp:CuHasO8}, $\Cu(A)$ also satisfies \axiomO{8}.
Now the result follows from \autoref{prp:IdealQuotient}.
\end{proof}

%==========================================================================================
\begin{rmk}
\label{rmk:RobRor}
Let us indicate an alternative proof of \autoref{prp:charDiv} that is based on the results by R{\o}rdam and Robert in \cite{RobRor13Divisibility}.

Let $A$ be a \ca.
By \autoref{prp:Morita}, $A$ is nowhere scattered if and only if its stabilization is.
We may therefore assume that $A$ is stable.

First, assume that $A$ is nowhere scattered.
Let $x\in\Cu(A)$ and $k\geq 2$.
Choose $a\in A_+$ with $x=[a]$, and let $B:=\overline{aAa}$ be the generated hereditary sub-\ca.
By \cite[Theorem~5.3(iii)]{RobRor13Divisibility}, $B$ has no irreducible representations of dimension less than or equal to $k-1$ if and only if $[a]$ is weakly $(k,\omega)$-divisible in $\Cu(B)$ (equivalently, in $\Cu(A)$).
Thus, it follows from \autoref{prp:firstChar} that $x$ is weakly $(k,\omega)$-divisible.

Conversely, assume that every element of $\Cu(A)$ is weakly $(2,\omega)$-divisible.
To reach a contradiction, assume that $A$ is not nowhere scattered.
We need to find $x\in\Cu(A)$ that is not weakly $(2,\omega)$-divisible.
Applying \autoref{prp:firstChar} we obtain a hereditary sub-\ca{} $B\subseteq A$ that admits a one-dimensional irreducible representation $\pi$.
Choose $c\in B_+$ with $\pi(b)\neq 0$, and consider the hereditary sub-\ca{} $C:=\overline{cAc}$.
Then $\pi$ restricts to a one-dimensional irreducible representation of $C$.
Using \cite[Theorem~5.3(iii)]{RobRor13Divisibility} again, it follows that $[c]$ is not weakly $(2,\omega)$-divisible in $\Cu(C)$, and therefore also not in $\Cu(A)$.
\end{rmk}

%==========================================================================================
%==========================================================================================
\section{Real rank zero and stable rank one}
\label{sec:rr0sr1}

%==========================================================================================
In this section, we establish characterizations of nowhere scatteredness among \ca{s} of real rank zero or stable rank one.

%==========================================================================================
A unital \ca{} has \emph{real rank zero} if the invertible, selfadjoint elements are dense in the set of selfadjoint elements.
A nonunital \ca{} has real rank zero if its minimal unitization does.
This important property is well-studied and enjoys many nice permanence properties.
We refer to \cite[p.453ff]{Bla06OpAlgs} and \cite{BroPed91CAlgRR0}.

%In \cite{PasRor07PIRR0} it is shown that a separable, purely infinite \ca{} has real rank zero if and only if its primitive ideal space has a basis of compact open sets and a $K$-theoretic obstruction vanishes.

The \emph{Murray-von Neumann semigroup} $V(A)$ of a \ca{} $A$ is defined as the set of equivalence classes of projections in $A\otimes\mathcal{K}$, where two projections $p$ and $q$ are equivalent if there exists a partial isometry $v\in A\otimes\mathcal{K}$ with $p=vv^*$ and $q=v^*v$. Equipped with the  addition induced by orthogonal sum and the algebraic order~$\leq_{\rm{alg}}$, $V(A)$ becomes a pre-ordered monoid.

A semigroup $S$ is \emph{weakly divisible} if for every $x\in S$ there exist $y,z\in S$ such that $x=2y+3z$.

%==========================================================================================
\begin{thm}
\label{prp:charDivRR0}
Let $A$ be a \ca{} of real rank zero.
Then the following are equivalent:
\begin{enumerate}
\item
$A$ is nowhere scattered;
\item
$V(A)$ is weakly divisible;
\item
$\Cu(A)$ is weakly divisible.
\end{enumerate}
\end{thm}
\begin{proof}
To show that~(1) implies~(2), assume that $A$ is nowhere scattered. By \autoref{prp:Morita}, we may assume $A$ to be stable. Thus, any element in $V(A)$ is of the form $[p]$ with $p$ a projection in $A$.

Then, given any $[p]\in V(A)$, it follows from \autoref{prp:firstChar} (7) that $pAp$ has no one-dimensional irreducible representation. As shown in the proof of \cite[Corollary~6.8]{AraGooPerSil10NonSimplePI}, this implies that $[p]=2y+3z$ for some $y,z\in V(A)$.

Assume now that (2) is satisfied.
In this case, if $x'\leq_{\rm{alg}} x$ and $x'=2y'+3z'$ in $V(A)$, then there exist $y,z\in V(A)$ with $y'\leq_{\rm{alg}} y$ and $z'\leq_{\rm{alg}} z$, and such that $x=2y+3z$. 
Indeed, given $w\in V(A)$ such that $x'+w=x$, we can find $u,v$ satisfying $w=2u+3v$. 
Setting $y=y'+u$ and $z=z'+v$, the result follows.
 
Since $A$ has real rank zero, there exists an order preserving, monoid morphism $\alpha\colon V(A)\to \Cu (A)$ with sup-dense image; see, for example, \cite[Theorem~5.7]{ThiVil22DimCu}. That is, every element $x\in\Cu (A)$ can be written as the supremum of an increasing sequence in $\alpha (V(A))$. By the remark above, this implies that 
\[
 x=\sup_n (2y_n + 3z_n)
\]
with $(y_n)_n, (z_n)_n$ increasing sequences in $\alpha (V(A))$.

Thus, we get $x=2\sup_n (y_n) + 3\sup_n (z_n)$, as desired.

Finally, to show that~(3) implies~(1), assume that $\Cu(A)$ is weakly divisible.
By \autoref{prp:charDiv} it suffices to show that every element in $\Cu(A)$ is weakly $(2,\omega)$-divisible.
So let $x',x\in\Cu(A)$ satisfy $x'\ll x$.
By assumption, we obtain $y,z\in\Cu(A)$ such that $x=2y+3z$.
Set $s:=x+y$.
Then, one gets $2s=2x+2y\leq x$ and $x' \leq x \leq 3y+3z = 3s$, as required.
\end{proof}

%==========================================================================================
\begin{rmk}
\label{rmk:charDivRR0}
If $A$ does not have real rank zero, then weak divisibility of $V(A)$ is not equivalent to nowhere scatteredness.

Take, for example, the Jiang-Su algebra $\mathcal{Z}$. Then, since $1\in\NN$ cannot be decomposed as $2x+3y$ in $\NN$, it follows that $V(\mathcal{Z})\cong\NN$ is not weakly divisible.

However, it follows from \autoref{exa:Sim} that $\mathcal{Z}$ is nowhere scattered.
\end{rmk}

%==========================================================================================
A unital \ca{} has \emph{stable rank one} if its invertible elements are dense.
A nonunital \ca{} has stable rank one if its minimal unitization does.
Just as real rank zero, the property of stable rank one enjoys many permanence properties.
We refer to \cite[Section~V.3.1]{Bla06OpAlgs}.

Cuntz semigroups of stable rank one \ca{s} have additional regularity characteristics;
see \cite{Thi20RksOps, AntPerRobThi22CuntzSR1}.
In particular, by \cite[Theorem~3.5]{AntPerRobThi22CuntzSR1}, they satisfy the \emph{Riesz Interpolation Property}:
If $x_1,x_2,y_1,y_2$ are such that $x_j\leq y_k$ for each $j,k\in\{1,2\}$, then there exists $z$ with $x_1,x_2\leq z\leq y_1,y_2$.

A \CuSgp{} $S$ is \emph{countably based} if it contains a countable subset $B$ such that every element in $S$ is the supremum of an increasing sequence of elements in~$B$.
If $S$ is a countably based \CuSgp{}, then every upward directed subset of $S$ has a supremum. Separable \ca{s} have countably based Cuntz semigroups.

Given a separable, stable rank one \ca{} $A$, it follows from  \cite[Theorem~3.8]{AntPerRobThi22CuntzSR1} that $\Cu (A)$ is \emph{inf-semilattice ordered}. That is, infima of finite sets exists and, for every triple $x,y,z$, one has 
\[
 (x+z)\wedge(y+z)=(x\wedge y)+z.
\]

%==========================================================================================
\begin{dfn}
\label{dfn:IA}
We say that a \CuSgp{} $S$ satisfies the \emph{interval axiom} if for all $x',x,y,u,v\in S$ satisfying
\[
x'\ll x, \quad
x \ll y+u, \andSep
x \ll y+v,
\]
there exists $w\in S$ such that
\[
x'\ll y+w, \andSep w\ll u,v.
\]
\end{dfn}

%==========================================================================================
\begin{rmk}
\label{rmk:IA}
The interval axiom as defined in \autoref{dfn:IA} above is the $\Cu$-version of the `algebraic interval axiom' considered by Wehrung for positively ordered monoids in \cite[Paragraph~1.3]{Weh96MonIntervals}.
\end{rmk}

%==========================================================================================
\begin{prp}
\label{prp:charInfSemiOrdered}
Let $S$ be a countably based \CuSgp{}.
Then $S$ is inf-semilattice ordered if and only if $S$ has the Riesz Interpolation Property and satisfies the interval axiom.
\end{prp}
\begin{proof}
To show that forward implication, assume that $S$ is inf-semilattice ordered.
Given $x_j,y_k\in S$ with $x_j\leq y_k$ for $i,j\in\{1,2\}$, we have $x_1,x_2\leq (y_1\wedge y_2)\leq y_1,y_2$, which shows that $S$ has the Riesz Interpolation Property.
To verify the interval axiom, let $x',x,y,u,v\in S$ satisfy
\[
x'\ll x, \quad
x \ll y+u, \andSep
x \ll y+v.
\]

Then
\[
x'\ll x \leq (y+u)\wedge(y+v)=y+(u\wedge v),
\]
which allows us to choose $w\in S$ such that $x'\ll y+w$ and $w\ll (u\wedge v)$. Thus, $w$ has the desired properties.

Let us show the backward implication.
Given $x,y\in S$, it follows from the Riesz Interpolation Property that the set $L:=\{z\in S : z\leq x,y\}$ is upward directed.
Since $S$ is countably based, the supremum of $L$ exists and $x\wedge y=\sup L$.
Thus, $S$ is an inf-semilattice.

To show that addition distributes over infima, let $x,y,z\in S$ and note that the inequality $(x+z)\wedge(y+z)\geq (x\wedge y)+z$ is clear. For the other inequality, set $w=(x+z)\wedge(y+z)$, and let $w'\in S$ satisfy $w'\ll w$.
Applying the interval axiom, we obtain $s$ such that $w'\leq z+s$ and $s\leq x,y$.
Then $s\leq x\wedge y$, and therefore $w'\leq (x\wedge y)+z$.
Since this holds for every $w'$ way-below $w$, we get $w\leq (x\wedge y)+z$.
\end{proof}

%==========================================================================================
A submonoid $T$ of a \CuSgp{} $S$ is said to be a \emph{sub-\CuSgp{}} if $T$ is a \CuSgp{} with the induced order, and if the inclusion $T\to S$ preserves suprema of increasing sequences and the way-below relation.

In analogy to its definition for \ca{s} (as defined in \autoref{sec:permanence}), we say that a property $\mathcal{P}$ for \CuSgp{s} satisfies the \emph{L{\"o}wenheim-Skolem condition} if for every \CuSgp{} with property $\mathcal{P}$ there exists a family $\mathcal{S}$ of countably based sub-\CuSgp{s} of $S$ each having property $\mathcal{P}$, and such that $\mathcal{S}$ is $\sigma$-complete and cofinal;
see \cite[Paragraph~5.2]{ThiVil21DimCu2}.
The next result can be proved with the methods that are used to to prove \cite[Proposition~5.3]{ThiVil21DimCu2}.
We omit the details.

%==========================================================================================
\begin{prp}
\label{prp:LS-IA-RIP}
The interval axiom and the Riesz Interpolation Property both satisfy the L{\"o}wenheim-Skolem condition.
\end{prp}

%==========================================================================================
If $A$ is a separable \ca{} of stable rank one, then $\Cu(A)$ is inf-semilattice ordered by \cite[Theorem~3.8]{AntPerRobThi22CuntzSR1}, and thus satisfies the interval axiom. 
With the techniques of \cite{ThiVil21DimCu2}, we can show that this also holds in the nonseparable case.

%==========================================================================================
\begin{cor}
\label{prp:CuSR1}
Let $A$ be a \ca{} of stable rank one.
Then $\Cu(A)$ satisfies weak cancellation, the Riesz Interpolation Property, and the interval axiom.
\end{cor}
\begin{proof}
It follows from \cite[Theorem~4.3]{RorWin10ZRevisited} and \cite[Theorem~3.5]{AntPerRobThi22CuntzSR1} that $\Cu(A)$ is weakly cancellative and has the Riesz Interpolation Property.
To show that $\Cu(A)$ satisfies the interval axiom, let $x',x,y,u,v\in \Cu(A)$ satisfy
\[
x'\ll x, \quad
x \ll y+u, \andSep
x \ll y+v.
\]

Applying \cite[Proposition~6.1]{ThiVil21DimCu2}, we obtain a $\sigma$-complete and cofinal collection of separable sub-\ca{s} $B\subseteq A$ such that the inclusion $B\to A$ induces an order-embedding $\Cu(B)\to\Cu(A)$ whose image contains $x',x,y,u,v$.
Using that stable rank one satisfies the L{\"o}wenheim-Skolem condition, we may choose such separable sub-\ca{s} $B\subseteq A$ with  stable rank one.
Thus, each Cuntz semigroup $\Cu(B)$ is inf-semilattice ordered by \cite[Theorem~3.8]{AntPerRobThi22CuntzSR1}.
It follows from \autoref{prp:charInfSemiOrdered} that $\Cu(B)$ satisfies the interval axiom.
Hence, an element $w$ with the desired properties can be found in $\Cu(B)$ and, since $\Cu(B)$ can be identified with a sub-\CuSgp{} of $\Cu (A)$, we deduce that $\Cu (A)$ satisfies the interval axiom.
\end{proof}

%==========================================================================================
\begin{rmk}
Cuntz semigroups of separable \ca{s} of real rank zero do not necessarily satisfy the interval axiom.
Indeed, in \cite{Goo96RR0K0NotRiesz} Goodearl constructs a separable (stably finite, nuclear) \ca{} $A$ of  real rank zero such that $K_0(A)$ does not have the Riesz decomposition property.
Thus, $K_0(A)$ does not have the Riesz interpolation property and, by \cite[Lemma~4.2]{Per97StructurePositive}, it follows that $V(A)$ does not have the Riesz interpolation property either.
However, as noted in \cite{Weh96MonIntervals}, if a refinement monoid satisfies the (algebraic) interval axiom, then it satisfies Riesz interpolation.
By \cite[Lemma~2.3]{AraPar96RefMonWkComparability}, which is based on \cite{Zha90DiagProjMultiplierCa}, $V(A)$ is a refinement monoid and, therefore,  $V(A)$ does not satisfy the algebraic interval axiom.
Using that $\Cu(A)$ is isomorphic to the sequential ideal completion of $V(A)$ (see for example \cite[Remark~5.5.6]{AntPerThi18TensorProdCu}), we deduce that $\Cu(A)$ does not satisfy the interval axiom.
\end{rmk}

%==========================================================================================
\begin{qst}
\label{qst:IntervalAxiom}
For which \ca{s} does the Cuntz semigroup satisfy the interval axiom?
\end{qst}

%==========================================================================================
\begin{lma}
\label{prp:preCharDivSR1}
Let $S$ be a weakly cancellative \CuSgp{} satisfying \axiomO{5}, the interval axiom, and the Riesz Interpolation Property, and let $x\in S$ and $k\geq 2$.
Then $x$ is weakly $(k,\omega)$-divisible if and only if $x$ is $(k,\omega)$-divisible.
\end{lma}
\begin{proof}
It suffices to show the forward implication.
Thus, assume that $x$ is weakly $(k,\omega)$-divisible.
Then there exists a sequence $(y_n)_n$ in $S$ such that $ky_n\leq x$ for each $n$, and such that $x\leq\sum_{n=1}^\infty y_n$.

By \cite[Propositions~5.3, 5.4]{ThiVil21DimCu2} and \autoref{prp:LS-IA-RIP}, the properties \axiomO{5}, weak cancellation, the interval axiom, and the Riesz Interpolation Property each satisfy the L{\"o}wenheim-Skolem condition.
Using this, we find a countably based sub-\CuSgp{} $T\subseteq S$ containing $x,y_1,y_2,\ldots,$ and such that $T$ satisfies \axiomO{5}, weak cancellation, the interval axiom, and the Riesz Interpolation Property.
By \autoref{prp:charInfSemiOrdered}, $T$ is inf-semilattice ordered.

We note that $x$ is weakly $(k,\omega)$-divisible in $T$.
Applying \cite[Theorem~5.5]{AntPerRobThi22CuntzSR1}, it follows that $x$ is $(k,\omega)$-divisible in $T$, and hence also in $S$.
\end{proof}

%==========================================================================================
\begin{thm}
\label{prp:charDivSR1}
Let $A$ be a \ca{} of stable rank one.
Then the following are equivalent:
\begin{enumerate}
\item
$A$ is nowhere scattered;
\item
every element in $\Cu(A)$ is $(2,\omega)$-divisible;
\item
every element in $\Cu(A)$ is $(k,\omega)$-divisible for every $k\geq 2$.
\end{enumerate}
\end{thm}
\begin{proof}
As observed in \autoref{pgr:Cu}, $\Cu(A)$ is a \CuSgp{} satisfying \axiomO{5}.
By \autoref{prp:CuSR1}, $\Cu(A)$ satisfies weak cancellation, the Riesz Interpolation Property, and the interval axiom.
Thus, \autoref{prp:preCharDivSR1} shows that an element in $\Cu(A)$ is weakly $(k,\omega)$-divisible if and only if it is $(k,\omega)$-divisible.
Now the result follows from \autoref{prp:charDiv}.
\end{proof}

%==========================================================================================
The question of whether this stronger divisibility property in $\Cu(A)$ characterizes nowhere scatteredness is connected to the Global Glimm Problem, which will be studied in \cite{ThiVil21arX:Glimm}.

%==========================================================================================
\providecommand{\href}[2]{#2}

\end{document}